\documentclass[12pt]{article}
\usepackage{amsmath,amsfonts,amssymb,amsthm,amscd,stmaryrd,mathabx, comment}
\usepackage[utf8]{inputenc}

\usepackage[margin=2.2cm]{geometry}
\usepackage{enumitem}
\setlist{nosep,leftmargin=*}
\usepackage{titling}
\title{Simple Models of Randomization and Preservation Theorems}
\date{\today}
\author{Karim Khanaki\thanks{Partially supported by IPM grant  1404030118}\\Arak University of Technology, and 
\\Institute for Research in Fundamental Sciences  (IPM)
 \and   Massoud Pourmahdian
\\Institute for Research in Fundamental Sciences  (IPM), and\\Amirkabir University of Technology}

\newtheorem{Theorem}{Theorem}[section]
\newtheorem{Proposition}[Theorem]{Proposition}
\newtheorem{Definition}[Theorem]{Definition}
\newtheorem{Remark}[Theorem]{Remark}
\newtheorem{Lemma}[Theorem]{Lemma}
\newtheorem{Corollary}[Theorem]{Corollary}
\newtheorem{Fact}[Theorem]{Fact}
\newtheorem*{claim*}{Claim}

\newcommand{\R}{\mathbb R}

\begin{document}
\maketitle

\begin{abstract} { The main purpose of this paper is to provide new, more uniform
model-theoretic and combinatorial proofs of the preservation
theorems established in \cite{Ben VC} and \cite{BK}. More precisely, we show that randomization preserves both
$NIP$ and stability: if a first-order theory $T$ has $NIP$, then
so does $T^R$, and if $T$ is stable, then so is $T^R$. 
Our proofs rely on the systematic use of simple models of $T^R$ and quantitative combinatorial methods. In the $NIP$ case, we use indiscernible arrays and Rademacher
mean width, whereas in the stable case, we use a Ramsey lemma
for cross-blocks and sequential Rademacher mean width. 
Using simple models of $T^R$ has the advantage of allowing us to re-prove these theorems in a simpler and quantitative manner.
  We finally turn our attention to $NSOP$ in randomization. We show that based on the definition of $NSOP$ given in \cite{K-sop}, $T^R$ is stable if and only if it is $NIP$ and $NSOP$.}

\end{abstract}

\section{Introduction}
The main concern of the current paper is the general question of whether certain model-theoretic properties are preserved under the randomization construction. 
Following the seminal paper of Keisler, \cite{Kei}, the randomization
of a first-order theory $T$, denoted by $T^R$, was introduced in
\cite{BK} as a complete continuous theory. 
 Several papers, including \cite{BK}, \cite{Ben VC}, \cite{AK},
\cite{AGK}, \cite{I}, and \cite{CT}, establish that various
model-theoretic properties, including $\omega$-categoricity,
$\omega$-stability, stability, and $NIP$, are preserved under
randomization.


Among these preservation results, the case of $NIP$ is particularly relevant to our approach. The preservation of $NIP$ under randomization, in the general continuous-logic setting, was established by Ben Yaacov in \cite{Ben VC}, using substantial combinatorial and analytic machinery, including continuous VC theory, the correspondence between random types and Keisler measures, and estimates involving Gaussian mean width. The general framework for the randomization of continuous theories was subsequently developed in \cite{BenYaacov2013}.

This raises the question of whether $NIP$ preservation can be proved more directly, using methods closer to standard model-theoretic practice. We answer this question by working directly with simple models of randomization, rather than relying on the correspondence between random types and Keisler measures. This approach gives a more concrete proof and also yields quantitative bounds.\footnote{More precisely, our argument does not require a separate finite-VC-dimension result for the class of expectation functions associated with a measurable family of function classes of uniformly bounded VC dimension. Instead, we work directly with simple models of randomization and use the Lebesgue dominated convergence theorem to control the VC dimension along the relevant increasing chain of function classes.}

The preservation of stability under randomization was originally proved by Ben Yaacov and Keisler (in \cite{BK}, for classical case) using an adaptation of Shelah's ranks and the definability of types. Our proof follows a different route, while retaining the guiding idea of working directly with simple models of randomization. The stable case requires a different combinatorial analysis from the NIP case: we use a Ramsey-theoretic lemma for cross-blocks and sequential Rademacher mean-width estimates in place of the indiscernible-array argument and the ordinary Rademacher mean-width estimates used for $NIP$. Thus, the two preservation proofs share a common model-theoretic framework but exploit distinct combinatorial mechanisms suited to the respective properties.

In addition to the preservation theorems, we study $NIP$ and stability in configurations of independent random variables. The corresponding results, Theorems \ref{main} and \ref{main-2-2}, are proved by Ramsey-theoretic arguments and do not rely on Rademacher complexity or learning-theoretic machinery. These theorems therefore provide a complementary part of our approach: whereas mean-width estimates are useful in the general preservation arguments, the independent-variable setting admits a more explicitly combinatorial treatment.



 

We further raise some issues concerning the $NSOP$ property.    Without a universally accepted definition of $NSOP$ for continuous logic,
we show that, in the context of randomization, the definition of $NSOP$
given in \cite{K-sop} is the weakest condition among the known concepts
that satisfy Shelah's theorem.  That is, $T^R$ is stable if and only if it is $NIP$ and $NSOP$.\footnote{The choice of $NSOP$ is relevant because the form of Shelah's theorem used here holds for randomization theories under this notion. By contrast, \cite{K-classification} discusses continuous theories outside the randomization setting for which the corresponding form of Shelah's theorem fails under the notion of SOP considered there.}

The motivation for this paper is not merely to provide alternative proofs of known preservation theorems, but to identify a more direct and flexible model-theoretic framework for studying which structural properties survive randomization. The original proofs of $NIP$  and stability preservation, while fundamental, rely on substantially different techniques: the former draws on VC-theoretic and analytic machinery, whereas the latter uses an adaptation of Shelah's ranks and the definability of types. We show that simple models of randomization provide a common setting in which both preservation results can be approached through concrete combinatorial configurations. In the $NIP$ case, this leads to a proof based on indiscernible arrays and Rademacher mean width, together with quantitative bounds; in the stable case, an analogous strategy uses a Ramsey lemma for cross-blocks and sequential Rademacher mean width. Moreover, the preservation results for independent random variables (i.e., Theorems~\ref{main} and \ref{main-2-2}) are established by Ramsey-theoretic arguments that do not depend on Rademacher complexity or learning-theoretic machinery. These complementary methods reveal which parts of the preservation phenomena are genuinely model-theoretic and which admit quantitative combinatorial formulations. The discussion of $NSOP$ further places the preservation of stability in a broader classification-theoretic context, using the notion of $NSOP$ adopted in this paper. Finally, the method's applicability to incomplete theories, and its potential extension to continuous theories, indicate that simple models are not merely a technical convenience: they offer a systematic framework for investigating preservation under randomization and for seeking analogous results for other model-theoretic properties.

This paper is organized as follows: In Section 2 we give notations and preliminaries about indiscernible sequences/arrays, randomization, and $VC$-theorem for $NIP$ theories. In Section 3 we present our main results about studying $NIP$, stability, 
and $NSOP$ in $T^R$.
In Section 4, we finally turn our attention to the randomization of continuous theories and present a proof sketch showing how one could extend the ideas from Section 3 showing that the $NIP$  property can be transferred to the randomization $T^R$ of a continuous theory $T$.

\section{Preliminaries}
In this section we review the basic materials which are essential for proving the main theorems.
Until Section $4$, we let $T$ be a classical first-order theory in a language $L$ and $M$ a model of $T$.
 To simplify the notation we assume that all partitioned formulas are of the form $\phi(x;y)$ where  $|x|=1$ and $y$ is a finite tuple of variables. 
As we also use theories within continuous setting, it would be useful to make this convention to consider the interpretation of $L$-formulas as a $\{0,1\}$-valued functions, that is, for $a,b\in M$, $\phi(x;y)=1$ if $\models\phi(a;b)$ and $\phi(a;b)=0$ otherwise.

\subsection{Indiscernible  Arrays}
The following notions of array indiscernibility play an important role in proving our main results.
\begin{Definition}\label{def type}
\em{Let $\phi(x;y)$ be an $L$-formula and $n\in\mathbb N$.
A $\phi$-$n$-formula $\Phi(x_1,\ldots,x_n)$ is an $L$-formula of the form  $\exists y\big(\bigwedge_{i\in E}\phi(x_i;y)\wedge\bigwedge_{i\in F}\neg\phi(x_i;y)\big)$ or $\forall y\big(\bigvee_{i\in E}\phi(x_i;y)\vee\bigvee_{i\in F}\neg\phi(x_i;y)\big)$, where $E,F$ are (arbitrary) disjoint subsets of $\{1,\ldots,n\}$.  Let $\Delta_{\phi-n}$ be the set of $\phi$-$n$-formulas.  For $\boldsymbol{b} =(a_1,\ldots,a_n)\in M^n$, the $\Delta_{\phi-n}$-type of $\boldsymbol{b}$, denoted by $tp_{\phi,n}(\boldsymbol{b})$, is the set of all $\phi$-$n$-formulas $\Phi(x_1,\ldots,x_n)$ such that $\models\Phi(\boldsymbol{b})$.}
\end{Definition}

Notice that, up to logical equivalence, the set $\Delta_{\phi-n}$ is finite and negation-closed.

\begin{Definition}\label{indiscernible}
\em{	Fix $n,k,m\in \mathbb N$ with $n\leq m$. Let $(a_{ij})_{i\leq m, j\leq k}$ be an $(m,k)$-array  of elements of $M$.
	We say that the array $(a_{ij})_{i\leq m, j\leq k}$ is $\phi$-$k$-$n$-indiscernible (over the empty set) if  for each $i_1<\cdots<i_n\leq m$ and $j_1<\cdots<j_n\leq  m$, and for any $l\leq k$,
	$${\rm tp}_{\phi,n}(a_{i_1l},\ldots,  a_{i_nl})={\rm tp}_{\phi,n}(a_{j_1l},\ldots,a_{j_nl}).$$}
\end{Definition}

For a given array $I=(a_{ij})_{i\leq N,j\leq k}$, we denote it as $I=(\bar{a}_{1},\dots,\bar{a}_{N})$ if $\bar{a}_i$, is the $i$-th row  of $I$, for each $i\leq N$.
For given $m\leq N$, a $(m,k)$-sub-array of $I$ is a subsequence $(\bar a_{i_1},\ldots,\bar a_{i_m})$ of $I$ of length $m$.

Since $\Delta_{\phi-n}$ is finite, the following fact easily follows from the finite Ramsey's theorem.
\begin{Fact}\label{Ramsey}
For each $n,k,m\in\mathbb N$ with $n\leq m$, there is a large number $N\in\mathbb N$
such that for every $(N,k)$-array $(\bar a_1,\ldots,\bar a_N)$  of elements of $M$ there exists a $\phi$-$k$-$n$-indiscernible 
 $(m,k)$-sub-array  $(\bar a_{i_1},\ldots,\bar a_{i_m})$ of $I$.
\end{Fact}
Besides the following lemma connects the above notion of indiscernibility with the shattering of a sequence inside $M$, it is one of the main ingredient of proving Corollary~\ref{Corollary}.

\begin{Lemma}\label{lemma}
	Let $n,m\in\mathbb N$ with $2n\leq m$ and $(a_i:i\leq m)$ a sequence of elements in $M$.  Assume that $(a_i:i\leq m)$ is $\phi$-$1$-$n$-indiscernible and there is an element $b\in M$ such that $\sum_{i=1}^{m-1}|\phi(a_{i+1};b)-\phi(a_{i};b)|\geq 2n-1$.  Then the sequence $(a_i:i\leq n)$ is shattered by $\phi(x,y)$ in $M$. That is, for any $I\subseteq\{1,\ldots,n\}$ there exists $b_I\in M$ such that $\models\phi(a_i;b_I)$ if  and only if $i\in I$.
\end{Lemma}
\begin{proof}
	The proof is essentially the same as the proof of  the direction (i)~$\Longrightarrow$~(ii) of Lemma~2.8 of \cite{K-Baire}. Notice that, the indiscernibility of $(a_i:i\leq m)$ and  the existence of $b$  guarantee that the set $\{a_1,\ldots,a_n\}$ is shattered by $\phi(x,y)$ in $M$.
\end{proof}

The following Ramsey lemma is used in the proof of
Theorem~\ref{main-2-2}. We use the finite Ramsey
theorem directly on the truth-value patterns between pairs of
blocks, recording both directions of each cross-block interaction.
The resulting homogeneous sub-array is sufficient for the
order-pattern argument in
Theorem~\ref{main-2-2}.

\begin{Fact}
\label{fact:cross-blocks}

Let \(\phi(x;y)\) be an \(L\)-formula, and let
\(k,\ell,m\geq 1\). There is an integer
\(N=N(\phi,k,\ell,m)\) such that the following holds.

Suppose
\[
\bar a_i=(a_{i,p})_{p=1}^{k},
\qquad
\bar b_i=(b_{i,q})_{q=1}^{\ell}
\qquad (1\leq i\leq N),
\]
where \(a_{i,p}\in M^{|x|}\) and
\(b_{i,q}\in M^{|y|}\).

Then there are indices
\[
i_1<\cdots<i_m
\]
and matrices \(U,V\in\{0,1\}^{k\times\ell}\) such that
for every \(u<v\leq m\), \(p\leq k\), and \(q\leq\ell\),
\[
\phi(a_{i_u,p};b_{i_v,q})=U_{pq}
\]
and
\[
\phi(a_{i_v,p};b_{i_u,q})=V_{pq}.
\]

\end{Fact}

\begin{proof}

For each pair \(i<j\), color the unordered pair \(\{i,j\}\)
by the ordered pair of matrices
\[
U^{i,j}_{pq}:=\phi(a_{i,p};b_{j,q}),
\qquad
V^{i,j}_{pq}:=\phi(a_{j,p};b_{i,q}).
\]
There are at most \(2^{2k\ell}\) colors.

By the finite Ramsey theorem, there is a homogeneous
subset of size \(m\). The two matrices are constant on
all pairs from this subset, as required.

\end{proof}

\subsection{Simple Models of Randomization}\label{simple}
Although we assume some familiarity with continuous logic, randomizations and Keisler measures (cf. \cite{BK} and \cite{S}), to make this  article self-contained, we will briefly sketch some basic materials of  the randomization of first-order theory $T$, denoted by $T^R$,  as a continuous first-order theory.
Recall from \cite{BK} $T^R$, is a complete continuous theory whose elements of models are random elements of $M$ for some  model $M$ of $T$.\footnote{Note that for an arbitrary (not necessarily complete) theory, \(T^R\) can also be defined and will not necessarily be a complete theory. We only deal with this case in Theorem~\ref{thm:arbitrary}.}


We consider a randomization of a model $M$ of $T$ as a two-sorted continuous structure. While the elements in the first sort $\bf K$ are random elements of $M$,  the elements in the second sort $\bf B$, on the other hand,  are events in an underlying probability space. So, the randomization language $L^R$ is a two-sorted continuous language with sorts $\bf K$ and $\bf B$, an $n$-ary function symbol $\llbracket\phi(\bar x)\rrbracket$ of sort ${\bf K}^n\to{\bf B}$ for each first-order formula $\phi(\bar x)\in L$ with $|\bar x|=n$, a $[0,1]$-valued unary predicate symbol $\mathbb P$ of sort $\bf B$ for probability, and the Boolean operations $\cap, \cup, ^c, \perp, \top$ of sort $\bf B$. The language $L^R$ also has distance predicates $d_{\bf B}$ of sort $\bf B$ and $d_{\bf K}$ of sort $\bf K$.

A pre-structure for $L^R$ is a pair $({\cal K}, {\cal B})$ where $\cal K$ is an interpretation of sort $\bf K$ and $\cal B$ is an interpretation of sort $\bf B$.
  For a model $M\models T$ and  an atomless finitely additive probability space $\Omega=(\Omega_0, \mathcal{B}_0,\mu_0)$ denote
\[ M_0^\Omega := \big\{f\colon \Omega_0 \to M: f  \text{ has  finite image, and for each } a\in M, f^{-1}(\{a\})\in\mathcal{B}_0\big\}.
\]
Then  $(M_0^{\Omega}, {\cal B}_0)$ is {\em a simple randomization} of $M$ and  an  $L^R$-pre-structure in which the sort  $\bf K$ is interpreted as $M_0^{\Omega}$ and the sort $\bf B$ as ${\cal B}_0$ respectively. Furthermore,  the measure ${\mathbb P}$ is given by ${\mathbb P}(B)=\mu_0(B)$ for each $B\in {\cal B}_0$,  and  $\llbracket\phi(\bar x)\rrbracket$ functions are $\llbracket\phi(f_1,\dots,f_n)\rrbracket=\mu_0(\{t\in\Omega_0:M\models \phi(f_1(t),\dots,f_n(t))\})$.
Its distance predicates are $d_{\bf K}(f,g)=\mu_0(\{t\in\Omega_0:f(t)\neq g(t)\})$ and $d_{\bf B}(B,C)=\mu_0(B\triangle C)$ where $\triangle$ is the symmetric difference operation.
 By Proposition~2.2 of \cite{BK}, there is a unique complete theory $T^R$ for $L^R$, called the {\em randomization of $T$}, such that for each model $N$ of $T$, $(M_0^{\Omega}, {\cal B}_0)$ is a pre-model of $T^R$. The simple randomizations of models of $T$ are called {\em simple pre-models of $T^R$}.

Every simple pre-model $(M_0^{\Omega},{\cal B}_0)$ leads to a model $M^\Omega$ of $T^R$, through the identification of distance zero elements in both sorts as well as completing both sorts with respect to their metrics. These structures are called {\em simple models} of $T^R$.

By Theorem~2.9 of \cite{BK}, $T^R$ is a complete continuous theory and enjoys the elimination of quantifiers in the language $L^R$. So any $L^R$-definable predicate is a uniform limit of boolean combinations  of $L^{R}$-formulas of the form $\mathbb{P}\llbracket\phi(x_1,...,x_n)\rrbracket$ where $\phi(x_1,\ldots,x_n)\in L$.

\medskip

Suppose that $f,g$ are elements of the simple randomization 
$(M_0^{\Omega}, {\cal B}_0)$. Then there are  finite  measurable partitions ${\cal A}_1,\ldots,{\cal A}_k$ and ${\cal B}_1,\ldots,{\cal B}_m$
 of $\Omega_0$ and elements $a_1,\ldots,a_k, b_1,\ldots,b_m\in M$ such that for any $i\leq k$ and all $t\in {\cal A}_i$, $f(t)=a_i$, and
  for any $j\leq m$ and all $t\in {\cal B}_j$, $g(t)=b_j$. 
 We say that $f$ and $g$ are {\bf independent}
 if  for every $i\leq k$, $j\leq m$,  we have $\mu_0({\cal A}_i\cap{\cal B}_j)=\mu_0({\cal A}_i)\cdot\mu_0({\cal B}_j)$.

\medskip
The significance of simple pre-structures is explained in the following proposition:
\begin{Proposition}\label{transfer}
Let $(M_0^{\Omega}, {\cal B}_0)$ be a simple randomization of $M$. Suppose that $f,g\in M_0^{\Omega}$ are independent, and $\phi(x;y)\in L$ with $|x|=|y|=1$. Then, there are $a_1,\ldots,a_k,b_1,\ldots,b_m\in M$, and $r_i,s_j\in[0,1]$ with $\sum_{i\leq k}r_i=\sum_{j\leq m}s_j=1$ such that
${\mathbb P}\llbracket\phi(f;g)\rrbracket=\sum_{j\leq m}\sum_{i\leq k}s_j\cdot r_i\cdot\phi(a_i;b_j)$.
\end{Proposition}
\begin{proof}
As $f\in M_0^{\Omega}$,  there are a finite  measurable partition ${\cal A}_1,\ldots,{\cal A}_k$ of $\Omega_0$ and elements $a_1,\ldots,a_k\in M$ such that for any $i\leq k$ and all $t\in {\cal A}_i$, $f(t)=a_i$. This suggests that,
 we consider $f$ as a formal sum $\sum_{i\leq k}a_i^{{\cal A}_i}$.
Similarly, we assume that $g$ is of the form $\sum_{j\leq m}{b_j}^{{\cal B}_j}$ where $b_1,\ldots,b_m\in M$ and ${\cal B}_1,\ldots,{\cal B}_m$ is a measurable partition of $\Omega_0$. Then, we have
$${\mathbb P}\llbracket\phi(f;g)\rrbracket={\mathbb P}\llbracket\phi(\sum_{i\leq k}a_i^{{\cal A}_i};\sum_{j\leq m}b_j^{{\cal B}_j})\rrbracket
=\sum_{j\leq m}\sum_{i\leq k}\mu_0({\cal B}_j\cap{\cal A}_i)\cdot{\mathbb P}\llbracket\phi(a_i;b_j)\rrbracket$$
$$
\overset{\text{*}}{=}
\sum_{j\leq m}\sum_{i\leq k}\mu_0({\cal B}_j)\cdot\mu_0({\cal A}_i)\cdot{\mathbb P}\llbracket\phi(a_i;b_j)\rrbracket$$
$$=\sum_{j\leq m}\sum_{i\leq k}\mu_0({\cal B}_j)\cdot\mu_0({\cal A}_i)\cdot\phi(a_i;b_j).$$
The equality $\overset{\text{*}}{=}$ follows from the independence of $f$ and $g$.
(To be more precise here, in the penultimate line, $a_i$ is the constant map $t\mapsto a_i$ from $\Omega_0$ to $\{a_i\}$. Similarly for $b_j$.)
\end{proof}

\begin{Remark}
Considering the notations in the proof of Proposition~\ref{transfer}, suppose that $f,g\in M_0^\Omega$ are of the forms $\sum_{i\leq k} a_i^{{\cal A}_i}$ and $b^{\Omega}$ respectively. Note that in this case, $f$ and $g$ are independent.
Let $\mu_{f}$ and $\nu_g$ be two Keisler measures on the variable $x$ and $y$ respectively, defined as follows:
$$\mu_f(\phi(x;c))=\sum_{i\leq k}\mu_0({\cal A}_i)\cdot\phi(a_i;c), \ \ \ \ \nu_g(\phi(c;y))=\phi(c;b)$$
 for any $c\in M$. 
Then, assuming  the random variables are independent,  
$${\mathbb P}\llbracket\phi(f;g)\rrbracket=\mu_f\otimes \nu_g(\phi(x;y)).$$

\end{Remark}
The above facts are part of the key components for proving Theorems~\ref{main} and \ref{main-2-2}.

\subsection{The $VC$-Theorem}
We recall the classical VC theorem, which will be used
in the proofs of our main results. 
The reader is referred to \cite{S} and \cite{Ben VC} for extensive accounts of $NIP$ formulas and  $VC$-dimension, in both classical and continuous logics.


\begin{Definition}\label{gen-stable}
{\em A Keisler measure $\mu(x)$ over  a model $N$ is called {\em generically stable} if for every formula $\phi(x;y)$ and $\epsilon>0$, there are $a_1,\ldots,a_n \in N$ such that: $$\sup_{b\in N}\big|\mu(\phi(x;b))-{\rm Ave}(a_1,\ldots,a_n;\phi(x;b))\big|\leq \epsilon.$$ }
\end{Definition}

Clearly, every generically stable measure is definable over and finitely satisfiable in a small set/model (cf. \cite[Theorem~7.29]{S}). On the other hand, notice that the measures of the form $\sum_{i=1}^k r_i\cdot\phi(a_i;y)$ (with $a_i\in M$ and $r_i\in[0,1], \sum_{i=1}^k r_i=1$) are generically stable.

\medskip
  By the $VC$-theorem (cf. \cite{S}, Theorem~6.6 and Lemma~7.24),  the number $n$ in Definition~\ref{gen-stable} depends only on $\phi(x;y)$ and $\epsilon$:

\begin{Fact}[\cite{S-note}, Fact~1.2]\label{fact_genstable}
	Let $\phi(x;y)$ be a  $NIP$ formula and $\epsilon>0$. Then there is an integer $n$, depending only on $\phi$ and $\epsilon$,   such that for any generically stable measure $\mu(x)$ on  $M$, there are $a_1,\ldots,a_n \in M$ such that $$\sup_{b\in M}|\mu(\phi(x;b))-{\rm Ave}(a_1,\ldots,a_n;\phi(x;b))|\leq \epsilon.$$
(In this case, sometimes for brevity, we write $|\mu(\phi)-{\rm Ave}(\bar a)(\phi)|\leq\epsilon$.)
\end{Fact}


\section{Preservation Theorems}
In this section, we present the main results. We first prove that the $NIP$ property is preserved under the randomization 
(for complete theories in Corollary~\ref{Corollary} and for arbitrary theories in Theorem~\ref{thm:arbitrary}).  We then  prove the preservation of stability for complete theories (Theorem~\ref{thm:stability-randomization-seq}). We finally discuss some results about $NSOP$ of $T^{R}$, showing that $T^R$ is stable if and only if it is $NIP$  and $NSOP$.

\subsection{Preservation of $NIP$}
Below we recall the $NIP$  in continuous logic. 
\begin{Definition}\label{NIP def}
\em{
Let $T_c$ be a complete continuous $\mathcal{L}$-theory  and $\varphi(x;y)$ a partitioned $\mathcal{L}$-formula.

\noindent
(i) We say that $\varphi(x;y)$ has $IP$ (independence property) if there are  $r<s\leq 1$ and a sequence $(a_i:i<\omega)$ of parameters (in the monster model)  such that for each $I\subseteq\mathbb N$, there is some $b$ such that $\varphi(a_i;b)\leq r$ if $i\in I$ and $\varphi(a_i;b)\geq s$ if $i\notin I$. We say that $\varphi(x;y)$ has $NIP$ if it has no $IP$.

\noindent
(ii) We say that $T_c$ has $IP$ if there is a formula $\varphi$ with the $IP$.  Otherwise we say that $T_c$ has the $NIP$.

\noindent

\noindent
}
\end{Definition}

\begin{Remark} \label{remark-com}
\em{
By compactness, a complete theory $T_c$ has the $IP$ if there are a formula $\varphi(x;y)$ and $r<s\leq 1$ such that for any $n\in\mathbb N$ and each (pre-)model $N$, there are $a_1,\ldots,a_n\in N$ such that for any $I\subseteq\{1,\ldots,n\}$ there is some $b_I\in N$ with $\varphi(a_i;b_I)\leq r$ if $i\in I$, and $\varphi(a_i,b_I)\geq s$ otherwise.

}\end{Remark}

From now on we restrict our attention to $T^R$. Notice that, by the elimination of quantifiers, to show that the theory $T^R$ has the $NIP$ we have to prove that the formulas of the form ${\mathbb P}\llbracket\phi(x;y)\rrbracket$ have $NIP$.

\begin{Definition}\label{random shattered} 
\em{Fix $N\in\omega+1$.\footnote{Our intention is to provide a definition that applies simultaneously to both finite and infinite sequences.}  Let $(\mu_i(x):i< N)$ be a sequence of (average) measures on $M$, $\phi(x;y)$ an  $L$-formula, and $r<s\leq 1$. We say that $(\mu_i:i< N)$ is shattered by the average elements in $M$ for $\phi(x;y)$ and $r<s$ if for any $I\subseteq\{i: i<N\}$ there exist $t\in{\mathbb N}$ and $b_1^I,\ldots,b_t^I\in M$ such that
$\frac{1}{t}\sum_{j=1}^t\mu_i(\phi(x;b_j^I))\leq r$ if $i\in I$, and  $\frac{1}{t}\sum_{j=1}^t\mu_i(\phi(x;b_j^I))\geq s$ otherwise.
 In this case, sometimes we simply say that the sequence is shattered by averages of elements in $M$. 

}
\end{Definition}

The following proposition justifies Definition~\ref{random shattered}.

\begin{Proposition}\label{translate}
 (Assuming $NIP$) Let $(f_i:i< N)$ be a sequence of elements in a simple randomization $M_0^{\Omega}$ of $M$. Suppose that for any $I\subseteq\{i:i<N\}$ there exists $g_I\in M_0^{\Omega}$ such that
 $g_I$ and $f_i$ ($i<N$) are independent 
 and 
  ${\mathbb P}\llbracket\phi(f_i;g_I)\rrbracket\leq r$ if $i\in I$ and ${\mathbb P}\llbracket\phi(f_i,g_I)\rrbracket\geq s$ otherwise. 
 Then for every positive \(\delta<|s-r|/2\) there is a sequence $(\mu_i:i<N)$ of average measures of elements in $M$ which is shattered by average elements in $M$ for $\phi(x;y)$ and $r+\delta<s-\delta$.
\end{Proposition}
\begin{proof}
Recall that the $f_i$'s and $g_I$ correspond to weighted measures.
By Proposition~\ref{transfer}, these weighted measures give rise to an \((r,s)\)-shattering of length \(N\).
By Fact~\ref{fact_genstable}, we can approximate these weighted measures with average measures.
Thus, we arrive at a shattering of average measures for \(r+\delta, s-\delta\) of length \(N\).
\end{proof}

\noindent
By Fact~\ref{fact_genstable} and Remark~\ref{remark-com}, 
we can assume that the $f_i$'s and $g_I$ are elements in a simple model $M^{\Omega}$. That is, there is a sequence of length $n$ of elements of $M^\Omega$, shattered by elements in $M^\Omega$ if and only if there is a sequence of length $n$ of elements of $M_0^{\Omega}$, shattered by elements in $M_0^{\Omega}$.

\medskip
We combine  Fact~\ref{Ramsey}, Lemma~\ref{lemma},  Proposition~\ref{transfer} and Fact~\ref{fact_genstable} to prove the following   main theorem.\footnote{After the first version of this paper was completed, we realized that Theorem~\ref{main} is essentially equivalent to Theorem~3.12 in Gannon's paper (Local Keisler measures and NIP formulas, The Journal of Symbolic Logic, 2019). 
 There, VC-theory (Gannon's Theorem~3.5 or equivalently Fact~\ref{fact_genstable} above) is also used. However, it seems that our approach here makes the discussion somewhat clearer and leads to a similar proof for the stable case (i.e., Theorem~\ref{main-2-2}  below).}
\begin{Theorem} \label{main}
Suppose $T$ is an  $NIP$ theory and  $M$ a model of $T$.
Then for each formula $\phi(x;y)$ and $r<s$, there is a natural number $N$ such that there is {\bf no} sequence $(\mu_1,\ldots,\mu_N)$ of average measures of elements in $M$ which is shattered by average elements in $M$ for $\phi(x;y)$ and $r<s$.
\end{Theorem}
\begin{proof}
Suppose on the contrary  that there are a formula $\phi(x;y)$ and $r<s$ such that for any large number $N$ there is a sequence $(\mu_1,\ldots,\mu_N)$ of average measures of elements in $M$ which is shattered by average elements in $M$ for $\phi$ and $r<s$.
That is, for any $I\subseteq\{1,\ldots,N\}$ there are $b_1^I,\ldots,b_t^I\in M$ such that
 $\frac{1}{t}\sum_{j=1}^t\mu_i(\phi(x;b_j^I))\leq r $ if $i\in I$, and $\frac{1}{t}\sum_{j=1}^t\mu_i(\phi(x;b_j^I))\geq s$ otherwise.

\medskip
Without loss of generality, using Fact~\ref{fact_genstable}, we can assume that there exists a fixed $k$ such that for every $i$ the measure $\mu_i(\phi(x;y))$ is of the form ${\rm Ave}(\bar a_i;\phi(x;y))$ where $\bar a_i=(a_{i,1},\ldots,a_{i,k})\in M^k$. Indeed, set $\epsilon=\frac{1}{8}|s-r|$, and using Fact~\ref{fact_genstable}, for $\phi(x;y)$ and a fixed $k$,  consider a sequence $(\nu_i:i\leq N)$ of average measures of the forms $\nu_i={\rm Ave}(a_{i,1},\ldots,a_{i,k};\phi(x;y))$ such that $|\mu_i(\phi)-\nu_i(\phi)|<\epsilon$. Set $r'=r+2\epsilon$ and $s'=s-2\epsilon$. Then, the sequence $(\nu_i:i\leq N)$ is shattered by elements in  $M_0^{\Omega}$ for $\phi$ and $r'<s'$. Thus, for consistency in notation, we  continue to  use $\mu_i$ and $r,s$.

Fix an even number \(m\geq 2n\). Let \(d=\operatorname{VC}(\varphi(x;y))\) and choose \(n=d+1\). By Fact~\ref{Ramsey}, for sufficiently large \(N\), there is a \(\varphi\)-\(k\)-\(n\)-indiscernible sub-array
\[
(\bar a_{i_1},\ldots,\bar a_{i_m})
\]
of \((\bar a_1,\ldots,\bar a_N)\), where \(i_1<\cdots<i_m\).

By the assumption, there are \(b_1,\ldots,b_t\in M\) such that
\[
\frac1t\sum_{j=1}^t\mu_{i_\ell}
   \bigl(\varphi(x;b_j)\bigr)\le r
\quad\text{if \(\ell\) is even},
\]
and
\[
\frac1t\sum_{j=1}^t\mu_{i_\ell}
   \bigl(\varphi(x;b_j)\bigr)\ge s
\quad\text{if \(\ell\) is odd}.
\]
Consequently,

\[
\begin{aligned}
(m-1)(s-r)
&\le
\sum_{\ell=1}^{m-1}
\left|
\frac1t\sum_{j=1}^t
\mu_{i_{\ell+1}}(\varphi(x;b_j))
-
\frac1t\sum_{j=1}^t
\mu_{i_\ell}(\varphi(x;b_j))
\right|\\[2mm]
&=
\sum_{\ell=1}^{m-1}
\left|
\frac1t\sum_{j=1}^t
\Bigl(
\mu_{i_{\ell+1}}(\varphi(x;b_j))
-
\mu_{i_\ell}(\varphi(x;b_j))
\Bigr)
\right|\\[2mm]
&\le
\sum_{\ell=1}^{m-1}
\frac1t\sum_{j=1}^t
\Bigl|
\mu_{i_{\ell+1}}(\varphi(x;b_j))
-
\mu_{i_\ell}(\varphi(x;b_j))
\Bigr|\\[2mm]
&=
\frac1t\sum_{j=1}^t\sum_{\ell=1}^{m-1}
\Bigl|
\mu_{i_{\ell+1}}(\varphi(x;b_j))
-
\mu_{i_\ell}(\varphi(x;b_j))
\Bigr|.
\end{aligned}
\]

Since \(n=d+1\) is larger than the VC-dimension of \(\varphi(x;y)\), Lemma~\ref{lemma} gives a constant \(K=2n-2=2d\), independent of \(m,t,j,h\), such that
\[
\sum_{\ell=1}^{m-1}
\left|
\varphi(a_{i_{\ell+1},h};b_j)
-
\varphi(a_{i_\ell,h};b_j)
\right|
\le K.
\]
It follows that
\[
(m-1)(s-r)
\le
\frac1{tk}\sum_{j=1}^t\sum_{h=1}^k K
=K.
\]
Since \(m\) can be chosen arbitrarily large, this contradicts \(s-r>0\). This proves the theorem.
\end{proof}
The reader should note that the preceding argument proves
the result only for formulas $\phi(x;y)$ with $|y|=1$.
However, the argument can be extended to tuples $y$ of
arbitrary finite length.

\begin{Remark}\label{Remark quantitative}
\em{ 
With the restriction that one of the variables \(x\) or \(y\) is interpreted only by elements in classical models such as \(M\), 
the argument presented in Theorem~\ref{main}
gives  a quantitative bound for  $VC$-dimension of ${\mathbb P}\llbracket\phi(x;y)\rrbracket$ in terms of $VC$-dimension of $\phi(x;y)$, the value $\epsilon=s-r$ and the Ramsey number (of $m,n,k\in\mathbb N$ where $m,n,k$ are given in Fact~\ref{Ramsey}).
Furthermore, our proof can be formulated in a purely combinatorial language without any  reference to specific model-theoretic language.}
\end{Remark}

\medskip
From now on, we aim to show that the previous theorem can be translated into the randomized setting, noting that the independence of the \( f_i \)'s and \( g_I \)'s in Proposition \ref{translate} imposes a limitation. 
Therefore, in order to resolve this issue, we introduce a new concept (namely, Rademacher complexity), and by combining well-known facts about this complexity with Theorem~3.5, we arrive at the desired result.

\section*{Rademacher Mean Width and VC-Dimension}
The following definitions and most of the material after them  are standard in learning theory; nevertheless, we present them fairly precisely. Our main goal is to arrive at Proposition ~\ref{independence}  using Theorem~\ref{main}   and these facts from learning theory.
\begin{Definition}
	Let $A \subseteq \R^n$. Let $\sigma$ be a randomly chosen vector in $\R^n$, 	which  $\sigma\in\{+1,-1\}^n$. We call $\mathbb{E}_\sigma[\sup_{a \in A} \sigma \cdot a]$ the \emph{Rademacher mean width}, denoted $w_R(A)$.
\end{Definition}

This definition will be used for a class of functions.

\begin{Definition} \label{Rad-def}
	Let $Q \subseteq [0,1]^X$ be a function class, and let $\bar x = (x_1,\dots, x_n) \in X^n$. Let $Q(\bar x) = \{(q(x_1),\dots,q(x_n)): q \in Q\}$. 
	Define the Rademacher mean width $r_Q(n)$ to be $\sup_{\bar x \in X^n}w_R(Q(\bar x))$.
\end{Definition}


The following fact is a well-known classical result in learning theory, which shows the relationship between $r_Q$ and the VC-dimension, and  fat-shattering dimension.

For a function class \(Q\subseteq [0,1]^X\), we use the
scale-sensitive VC-dimension \(vc_\epsilon(Q)\) and the
fat-shattering dimension \(fs_\epsilon(Q)\) in the sense of
Anderson~\cite{Anderson}. More precisely, \(vc_\epsilon(Q)\) is
the supremum, over \(0\leq r\leq 1-\epsilon\), of the VC-dimension
of the corresponding \((r,r+\epsilon)\)-threshold class, while
\(fs_\epsilon(Q)\) is the largest cardinality of a finite subset of
\(X\) that is \(\epsilon\)-fat-shattered by \(Q\). We use these
standard scale-sensitive notions throughout; in particular, the
quantitative relations between \(vc_\epsilon(Q)\),
\(fs_\epsilon(Q)\), and the covering numbers of \(Q\) are those
given in~\cite[Definition~2.16 and Facts~2.17--2.18]{Anderson}. In the following, we call a class \(Q\) a VC class if \(vc_\epsilon\) is finite for every \(\epsilon\). 
\begin{Fact}[Lemma 2.11, \cite{Anderson}] \label{vc-rademacher}
	Let $X$ be a set, and $Q \subseteq [0,1]^X$. The following are equivalent:
	\begin{enumerate}
		\item[(i)] $Q$ is a VC class.
		\item[(ii)] $\lim_{n \to \infty} \frac{r_Q(n)}{n} = 0$
	\end{enumerate}
	Moreover, the implication \[ \mathrm{VC}(Q)<\infty \quad\Longrightarrow\quad \frac{r_Q(n)}{n}\longrightarrow 0 \] is quantitative: for every \(0\le r<s\le 1\), an upper bound on the \(\mathrm{VC}\)-dimension of \(Q\) together with \(s-r\) gives a uniform bound on \(\frac{r_Q(n)}{n}\), tending to \(0\) as \(n\to\infty\).
\end{Fact}

The following remark demonstrates precisely the ``moreover'' part of the previous fact.
Indeed, all the material preceding  Proposition~\ref{independence}, with the possible exception of Lemmas~\ref{R-bound},  \ref{lem:equal-weight-approx} consists of standard facts from learning theory, and readers familiar with these results may skip this
material.
\begin{Remark}[Uniform quantitative form of (i)$\Longrightarrow$(ii) of Fact 3.9]
\label{rem:uniform-rademacher}
Moreover, the convergence in Fact~3.9 is uniform for classes with
uniform bounds on their scale-sensitive VC-dimensions. More precisely,
suppose that for every \(\epsilon>0\) there is \(d(\epsilon)<\infty\) such
that
\[
vc_\epsilon(Q)\leq d(\epsilon).
\]
Then, for every \(0<\eta<1\), there is
\[
N=N(\eta,d(\cdot))
\]
such that
\[
\frac{r_Q(n)}{n}<\eta
\qquad\text{for all }n\geq N
\]
and for every class \(Q\) satisfying the above bounds.

Indeed, fix \(0<\delta<1\). By Definition 2.16 and
Facts 2.17--2.18 of~\cite{Anderson}, a uniform bound on
\(vc_{\delta/4}(Q)\) yields a uniform bound on the covering numbers
of \(Q\). More precisely, since
\[
vc_{\delta/4}(Q)\leq d(\delta/4),
\]
Fact 2.17 gives
\[
fs_{\delta/4}(Q)
\leq
D_\delta,
\qquad
D_\delta:=
\left(2\left\lceil\frac{4}{\delta}\right\rceil-1\right)
d(\delta/4).
\]
Hence, by Fact 2.18, there is a constant
\[
C_{\delta,d(\delta/4)}<\infty,
\]
depending only on \(\delta\) and \(d(\delta/4)\), such that
\[
\log N_{Q,\delta}(n)
\leq
C_{\delta,d(\delta/4)}(\log n)^2
\]
for all \(n\geq2\).

Fix a sample \(\bar{x}=(x_1,\ldots,x_n)\in X^n\), and let \(G\) be a
\(\delta\)-net for \(Q(\bar{x})\) with respect to the \(l_\infty\)-metric.
Then
\[
|G|\leq N_{Q,\delta}(n).
\]
For every \(q\in Q\), choose \(g\in G\) such that
\[
\max_{1\leq i\leq n}
|q(x_i)-g(x_i)|\leq\delta.
\]
It follows that
\[
\begin{aligned}
\mathbb{E}_{\sigma}
\sup_{q\in Q}
\sum_{i=1}^n\sigma_iq(x_i)
&\leq
n\delta+
\mathbb{E}_{\sigma}
\max_{g\in G}
\sum_{i=1}^n\sigma_i g(x_i)\\
&\leq
n\delta+
\sqrt{2n\log |G|}.
\end{aligned}
\]
Taking the supremum over \(\bar x\in X^n\), we obtain
\[
r_Q(n)
\leq
n\delta+\sqrt{2n\log N_{Q,\delta}(n)}.
\]
Hence
\[
\frac{r_Q(n)}{n}
\leq
\delta+
\sqrt{\frac{2\log N_{Q,\delta}(n)}{n}}
\leq
\delta+
\sqrt{2C_{\delta,d(\delta/4)}}\,
\frac{\log n}{\sqrt n}.
\]
Now fix \(0<\eta<1\) and put \(\delta=\eta/2\). Since
\[
\frac{\log n}{\sqrt n}\longrightarrow 0,
\]
there is an integer
\[
N=N(\eta,d(\cdot))
\]
such that, for every \(n\geq N\),
\[
\sqrt{2C_{\eta/2,d(\eta/8)}}\,
\frac{\log n}{\sqrt n}
<
\frac{\eta}{2}.
\]
Consequently,
\[
\frac{r_Q(n)}{n}<\eta
\]
for every \(n\geq N\) and every class \(Q\) satisfying
\[
vc_\epsilon(Q)\leq d(\epsilon)
\qquad\text{for all }\epsilon>0.
\]
\end{Remark}

The following lemma quantitatively expresses the implication (ii) \(\Rightarrow\) (i) of Fact~\ref{vc-rademacher}.
\begin{Lemma}[Shattering implies a lower bound for the Rademacher mean width]
\label{rad-shattering-lower}
Let \(Q\subseteq [0,1]^X\), and let \(0\le r<s\le 1\).
If \(Q\) \((r,s)\)-shatters
\(x_1,\ldots,x_m\in X\), then
\[
r_Q(m)\geq \frac{s-r}{2}\,m.
\]
Equivalently, if
\[
\frac{r_Q(m)}{m}<\frac{s-r}{2},
\]
then \(Q\) does not \((r,s)\)-shatter \(m\) points.
\end{Lemma}

\begin{proof}
Put
\[
t:=\frac{r+s}{2}
\qquad\text{and}\qquad
\delta:=\frac{s-r}{2}.
\]
Since \(Q\) \((r,s)\)-shatters
\(x_1,\ldots,x_m\), for every
\(\sigma=(\sigma_1,\ldots,\sigma_m)\in\{-1,+1\}^m\)
there is \(q_\sigma\in Q\) such that
\[
\sigma_i\bigl(q_\sigma(x_i)-t\bigr)\geq\delta
\qquad\text{for every }i=1,\dots,m.
\]
Hence
\[
\begin{aligned}
\sup_{q\in Q}\sum_{i=1}^m\sigma_iq(x_i)
&\geq
\sum_{i=1}^m\sigma_iq_\sigma(x_i)\\
&=
t\sum_{i=1}^m\sigma_i
+
\sum_{i=1}^m
\sigma_i\bigl(q_\sigma(x_i)-t\bigr)\\
&\geq
t\sum_{i=1}^m\sigma_i+\delta m.
\end{aligned}
\]
Taking expectation with respect to the independent Rademacher
variables \(\sigma_1,\ldots,\sigma_m\), we obtain
\[
\mathbb E_\sigma
\left[
\sup_{q\in Q}\sum_{i=1}^m\sigma_iq(x_i)
\right]
\geq
\delta m,
\]
because
\[
\mathbb E_\sigma\left[\sum_{i=1}^m\sigma_i\right]=0.
\]
By the definition of \(r_Q(m)\),
\[
r_Q(m)\geq\delta m
=\frac{s-r}{2}\,m.
\]
This proves the lemma.
\end{proof}
We recall that Remark~\ref{rem:uniform-rademacher} and Lemma~\ref{rad-shattering-lower} above are not necessary for complete theories and are used only to prove the preservation of NIP for arbitrary theories.

\medskip
In the following, we connect the notion of the  Rademacher mean width with our main question.

Fix a formula \(\varphi(x;y)\). Set \(X=M_0^\Omega\) and
\[
Q=\{\phi_a:a\in M\},
\]
where
\[
\phi_a(f)=\mathbb P\llbracket\varphi(f;a^\Omega)\rrbracket.
\]

Fix a family of equal-measure partitions
\[
\mathcal B_k=(B_{k,1},\ldots,B_{k,k})
\qquad (k\geq 1)
\]
of \(\Omega\), with
\[
\mathbb P(B_{k,j})=\frac1k,
\]
such that whenever \(k\mid \ell\), the partition
\(\mathcal B_\ell\) refines \(\mathcal B_k\).
Such a family can be chosen on an atomless probability space.

For \(a\in M\) and an event \(B\), put
\[
\phi_{a,B}(f)
=\mathbb P\bigl(
\llbracket\varphi(f;a^\Omega)\rrbracket\cap B
\bigr).
\]
For each \(k\geq 1\), define
\[
Q_k=
\left\{
\sum_{j=1}^{k}\phi_{a_j,B_{k,j}}
:\ a_1,\ldots,a_k\in M
\right\}.
\]
Put
\[
\widehat Q=\bigcup_{k\geq 1}Q_{k!}.
\]

\begin{Lemma} \label{R-bound} Under the above assumptions and notation,
 \(r_{Q_k}(n) \leq r_Q(n)\) for all \(n\).
\end{Lemma}
\begin{proof}
For simplicity, and without loss of generality, we may assume \(k=2\).  
(Anywhere in the proof where \(2\) appears, it can be replaced by \(k\) to obtain the general case.) Assume that \(\mu(B) = \mu(B^c) = 1/2\). Note that the following relation holds for every \(a\) and \(f\): $\phi_a(f)=\phi_{a,B}(f)+\phi_{a,B^c}(f)$.
We also observe that:
 $\sup_{c,d}(\phi_{c,B}(f)+\phi_{d,B^c}(f))=\sup_c\phi_{c,B}(f)+\sup_d\phi_{d,B^c}(f)$.
Therefore, by the linearity of \(\mathbb{E}_\sigma\), we have: 
\(\mathbb{E}_\sigma[\sup_c\sum_1^n\sigma_i\phi_{c,B}(f_i)+
\sup_d\sum_1^n\sigma_i\phi_{d,B^c}(f_i)]=
\mathbb{E}_\sigma
[\sup_c\sum_1^n\sigma_i\phi_{c,B}(f_i)]
+\mathbb{E}_\sigma[\sup_d\sum_1^n\sigma_i\phi_{d,B^c}(f_i)]\)
 for any $\bar f=(f_1,\ldots,f_n)\in M_0^\Omega$.
Therefore, $r_{Q'}(n)\leqslant r_{Q_1}(n)+r_{Q_2}(n)$ for all $n$.

We will show that $r_{Q_1}\leqslant \frac{1}{2}r_Q$ and similarly $r_{Q_2}\leqslant \frac{1}{2}r_Q$. To summarise, $r_{Q'}\leqslant r_Q$.

For this, define $g_{a,f_i}(\omega):=1_{\{\phi(f_i(\omega),a) \text{ is true}\}}$.
Define the measure $\mu'$ via $\mu'(A):=2\mu(A\cap B)$ for any event $A$.  Then, $\phi_{a,B}(f_i)=\frac{1}{2}\int_\Omega g_{a,f_i}d\mu'$, and by linearity 
$\mathbb{E}_\sigma
[\sup_a\sum_1^n\sigma_i(\frac{1}{2}\int g_{a,f_i}d\mu')]
=\frac{1}{2}\mathbb{E}_\sigma
[\sup_a\sum_1^n\sigma_i(\int g_{a,f_i}d\mu')]$
 for any $\bar f$.

\begin{claim*}  For each \(i \leq n\), there exists \(f_i' \in M_0^\Omega\) such that for every \(a \in M\), \[ \int_\Omega g_{a,f_i} \, d\mu = \int_\Omega g_{a,f_i'} \, d\mu'. \] \end{claim*} 
\begin{proof}[Proof of Claim] 
Recall that \(\mu'(A)=2\mu(A\cap B)\) for every event \(A\), where \(B\subseteq\Omega\) is a measurable set with \(\mu(B)=\frac12\). For \(i\leq n\), write the simple random element \(f_i\) in the form
\[
f_i=\sum_{\ell=1}^{m_i} c_{i,\ell}^{A_{i,\ell}},
\]
where \(A_{i,1},\ldots,A_{i,m_i}\) is a measurable partition of \(\Omega\) and \(c_{i,1},\ldots,c_{i,m_i}\in M\). Since \(\mu'\) is atomless and
\[
\sum_{\ell=1}^{m_i} \mu(A_{i,\ell}) = 1,
\]
there is a measurable partition
\[
C_{i,1},\ldots,C_{i,m_i}
\]
of \(\Omega\) such that
\[
\mu'(C_{i,\ell}) = \mu(A_{i,\ell}) \qquad(\ell=1,\ldots,m_i).
\]
Define
\[
f_i' := \sum_{\ell=1}^{m_i}c_{i,\ell}^{C_{i,\ell}} \in M_0^\Omega.
\]
Then, for every \(a\in M\),
\[
\begin{aligned}
\int_\Omega g_{a,f_i}\,d\mu
&= \sum_{\ell=1}^{m_i} \mu(A_{i,\ell})\, \mathbf 1_{\{M\models\varphi(c_{i,\ell},a)\}} \\
&= \sum_{\ell=1}^{m_i} \mu'(C_{i,\ell})\, \mathbf 1_{\{M\models\varphi(c_{i,\ell},a)\}} \\
&= \int_\Omega g_{a,f_i'}\,d\mu'.
\end{aligned}
\]
Thus the desired equality holds.
\end{proof}
To summarise, since the functions \(f_i\) are simple, constructing the corresponding \(f_i'\) from the \(f_i\) is straightforward, and conversely, one can recover the \(f_i\) from the \(f_i'\). Therefore, $$r_{Q_1}(n)=
\frac{1}{2}\sup_{\bar f'}\mathbb{E}_\sigma
[\sup_a\sum_{i=1}^n\sigma_i(\int g_{a,f_i'}d\mu')]=
$$
$$\frac{1}{2}\sup_{\bar f}\mathbb{E}_\sigma
[\sup_a\sum_{i=1}^n\sigma_i(\int g_{a,f_i}d\mu)]
=\frac{1}{2}r_Q(n).
$$
A similar argument holds for \(Q_2\). Therefore, this suffices for the proof.
\end{proof}
 Note that in the above lemma, the given bound is independent of $k$ and hence holds uniformly for every $k \ge 2$. Therefore, by Theorem \ref{main}, assuming $T$ is $NIP$, since $Q$ is a VC class, it follows that $Q'$ is also a VC class.
 
\medskip
 We now only need the following simple lemma:
 
 \begin{Lemma} \label{R-bound-2}
Let \( Q_1 \subseteq Q_2 \subseteq \dots \) be an increasing sequence of function classes from \( X \) to \([0,1]\), and let \( \hat{Q} = \bigcup_{n=1}^\infty Q_n \). Then for any sample size \( k \geq 1 \),
\[
r_{\hat{Q}}(k) = \sup_{n \geq 1} r_{Q_n}(k).
\]
\end{Lemma}

\begin{proof}
Fix a sample
\[
\bar x=(x_1,\ldots,x_k)\in X^k.
\]
For each \(n\geq 1\), define
\[
Z_n(\sigma):=
\sup_{f\in Q_n}\sum_{i=1}^k \sigma_i f(x_i),
\qquad
\sigma\in\{-1,1\}^k.
\]
Since \(Q_n\subseteq Q_{n+1}\), we have
\[
Z_n(\sigma)\leq Z_{n+1}(\sigma)
\]
for every \(\sigma\). Moreover,
\[
-k\leq Z_n(\sigma)\leq k,
\]
because \(f(x_i)\in[0,1]\) for every \(f\in Q_n\).

Since
\[
\widehat Q=\bigcup_{n=1}^\infty Q_n,
\]
we have, for every \(\sigma\),
\[
\lim_{n\to\infty} Z_n(\sigma)
=
\sup_{f\in\widehat Q}\sum_{i=1}^k \sigma_i f(x_i).
\]
Thus, by the dominated convergence theorem,
\[
\begin{aligned}
w_R(\widehat Q(\bar x))
&=
\mathbb E_\sigma\left[
\sup_{f\in\widehat Q}
\sum_{i=1}^k \sigma_i f(x_i)
\right] \\
&=
\lim_{n\to\infty}
\mathbb E_\sigma\left[
\sup_{f\in Q_n}
\sum_{i=1}^k \sigma_i f(x_i)
\right] \\
&=
\sup_{n\geq 1} w_R(Q_n(\bar x)).
\end{aligned}
\]
Taking the supremum over all samples \(\bar x\in X^k\), we obtain
\[
\begin{aligned}
r_{\widehat Q}(k)
&=
\sup_{\bar x\in X^k}w_R(\widehat Q(\bar x)) \\
&=
\sup_{\bar x\in X^k}\sup_{n\geq1}w_R(Q_n(\bar x)) \\
&=
\sup_{n\geq1}\sup_{\bar x\in X^k}w_R(Q_n(\bar x)) \\
&=
\sup_{n\geq1}r_{Q_n}(k).
\end{aligned}
\]
\end{proof}

The following lemma is a simple fact about approximation in randomization, which we state and prove for the sake of completeness of the discussion.\footnote{Although the idea is simple, its statement and proof are lengthy, and a reader who is certain of its correctness may skip it.}

\begin{Lemma}[Simultaneous approximation by nested equal-measure partitions]
\label{lem:equal-weight-approx}

Let $(\Omega,\mathcal F,\mathbb P)$ be an atomless probability
space whose probability algebra is separable with respect to
the metric
\[
d_{\mathbb P}(A,B)=\mathbb P(A\mathbin{\triangle}B).
\]

Let $M_0^\Omega$ be a simple randomization over this probability
space.

There exists a fixed sequence of measurable partitions
\[
\mathcal B_{2^k}
=
(B_{2^k,1},\ldots,B_{2^k,2^k}),
\qquad k\geq 0,
\]
such that
\[
\mathbb P(B_{2^k,j})=2^{-k}
\qquad (1\leq j\leq 2^k),
\]
each $\mathcal B_{2^{k+1}}$ refines $\mathcal B_{2^k}$, and the
algebra generated by these partitions is dense in the probability
algebra with respect to $d_{\mathbb P}$.

Moreover, for every finite family
$g_1,\ldots,g_m\in M_0^\Omega$ and every $\epsilon>0$,
there exist $k\geq 0$ and
$g_1',\ldots,g_m'\in M_0^\Omega$ such that each $g_i'$ is
constant on every cell of $\mathcal B_{2^k}$ and
\[
d_K(g_i,g_i')<\epsilon
\qquad (1\leq i\leq m),
\]
where
\[
d_K(g,g')
=
\mathbb P\bigl(\{\omega\in\Omega:g(\omega)\neq g'(\omega)\}\bigr).
\]

In particular, the same partition $\mathcal B_{2^k}$ works
simultaneously for all $g_i'$.

Furthermore, for every $L$-formula $\varphi(x;y)$ and every
$f\in M_0^\Omega$,
\[
\left|
\mathbb P\llbracket\varphi(f;g_i)\rrbracket
-
\mathbb P\llbracket\varphi(f;g_i')\rrbracket
\right|
\leq d_K(g_i,g_i')
<\epsilon.
\]

\end{Lemma}

\begin{proof}

Since the probability algebra is atomless and separable, it is
isomorphic, as a probability algebra, to the Lebesgue probability
algebra on $[0,1]$. Let $\mathcal B_{2^k}$ be the inverse images,
under such an isomorphism, of the standard dyadic partitions
of $[0,1]$ into $2^k$ intervals of length $2^{-k}$.

These partitions are nested, and the algebra generated by them
is dense in the probability algebra with respect to
$d_{\mathbb P}$.

Fix $g_1,\ldots,g_m\in M_0^\Omega$ and $\epsilon>0$.
Since the $g_i$ are simple random variables, there is a finite
measurable partition
\[
A_1,\ldots,A_\ell
\]
of $\Omega$ such that every $g_i$ is constant on every $A_j$.
Write
\[
g_i\upharpoonright A_j=a_{i,j}\in M
\qquad
(1\leq i\leq m,\ 1\leq j\leq\ell).
\]

By the density established above, choose
$C_1,\ldots,C_\ell$ in the algebra generated by the dyadic
partitions such that
\[
\sum_{j=1}^{\ell}
\mathbb P(A_j\mathbin{\triangle}C_j)<\epsilon.
\]

Replace the $C_j$ by the disjoint sets
\[
D_1=C_1,\qquad
D_j=C_j\setminus\bigcup_{h<j}C_h
\quad (2\leq j<\ell),
\]
and
\[
D_\ell=\Omega\setminus\bigcup_{j<\ell}D_j.
\]
Then $D_1,\ldots,D_\ell$ form a measurable partition, and
each $D_j$ belongs to the algebra generated by the dyadic
partitions.

Put
\[
E=\bigcup_{j=1}^{\ell}
(A_j\mathbin{\triangle}C_j).
\]
Outside $E$, membership in each $A_j$ agrees with membership
in $C_j$. Since $(A_j)_{j=1}^{\ell}$ is a partition, outside
$E$ exactly one of the $C_j$ occurs, and the corresponding
$D_j$ has the same index. Consequently, the original tuple
$(g_1,\ldots,g_m)$ agrees with the tuple defined by assigning
$a_{i,j}$ on $D_j$, outside $E$.

Since there are finitely many $D_j$ and the dyadic algebras
are increasing, there is a single $k\geq0$ such that
\[
D_j\in\sigma(\mathcal B_{2^k})
\qquad (1\leq j\leq\ell).
\]

Define
\[
g_i'=\sum_{j=1}^{\ell}a_{i,j}^{D_j}
\qquad (1\leq i\leq m).
\]
Each $g_i'$ is constant on every cell of $\mathcal B_{2^k}$.
Moreover,
\[
\begin{aligned}
d_K(g_i,g_i')
&\leq \mathbb P(E)\\
&\leq
\sum_{j=1}^{\ell}
\mathbb P(A_j\mathbin{\triangle}C_j)\\
&<\epsilon
\end{aligned}
\]
for every $i=1,\ldots,m$. This proves the simultaneous
approximation assertion.

Finally, the events
$\llbracket\varphi(f;g_i)\rrbracket$ and
$\llbracket\varphi(f;g_i')\rrbracket$ can differ only on the
event where $g_i$ and $g_i'$ differ. Hence
\[
\left|
\mathbb P\llbracket\varphi(f;g_i)\rrbracket
-
\mathbb P\llbracket\varphi(f;g_i')\rrbracket
\right|
\leq d_K(g_i,g_i')<\epsilon.
\]
This completes the proof.

\end{proof}

The above lemmas and our Theorem~\ref{main} lead us to the desired final proposition:
\begin{Proposition}
\label{independence}
Fix an $L$-formula $\varphi(x;y)$. Suppose that $T$ is an NIP theory, $M\models T$, and $M_0^\Omega$ is a simple randomization of $M$. For each $r<s$, there is a natural number $N=N_{\varphi,r,s}$ such that there is no $f_1,\ldots,f_N\in M_0^\Omega$ such that for any $I\subseteq\{1,\ldots,N\}$ there is some $g_I\in M_0^\Omega$ with the following property:
$$\mathbb P\llbracket\phi(f_i;g_I)\rrbracket\leq r, \ \ \ \text{ if } \ i\in I, \text{ and }
$$
$$\mathbb P\llbracket\phi(f_i;g_I)\rrbracket\geq s, \ \ \ \text{ if } \ i\notin I. \ \ \ \ (\dagger) 
$$

Moreover, the bound can be chosen uniformly over
complete theories. More precisely, for every integer
$d\geq0$ and every $\epsilon>0$, there exists an
integer $N(d,\epsilon)$ with the following property:

For every complete first-order theory $S$, every
formula $\chi(x;y)$ satisfying
\[
\operatorname{VC}_S(\chi)\leq d,
\]
every simple randomization of a model of $S$,
and all $0\leq r<s\leq1$ with $s-r=\epsilon$,
the predicate
\[
\mathbb P\llbracket\chi(x;y)\rrbracket
\]
cannot $(r,s)$-shatter $N(d,\epsilon)$ points.


We emphasize that this quantitative bound is uniform
over complete theories. Indeed, the VC sampling
estimate used in Fact~\ref{fact_genstable} can be
chosen with sample size depending only on the
VC dimension of the formula and the approximation
error. The finite Ramsey bounds used above depend
only on the resulting finite parameters.

Together with
Remark~\ref{rem:uniform-rademacher}
and Lemma~\ref{rad-shattering-lower},
this yields a bound $N(d,\epsilon)$ independent
of the particular complete theory and of the
particular formula, except through its VC
dimension $d$ and the threshold gap $\epsilon$.
\end{Proposition}

\begin{proof} Since we want to use Lemma~\ref{lem:equal-weight-approx}, without loss of generality we assume that \(\Omega\) is separable.\footnote{Note that it is enough to prove the condition for one simple randomization model, so this does not impose any restriction. See  Corollary~\ref{Corollary}.}
Put $\varepsilon=s-r>0$ and let
\[
\eta=\frac{\varepsilon}{8}.
\]
We argue by contradiction. Suppose that for arbitrarily large
$N$ there are $f_1,\ldots,f_N\in M_0^\Omega$ such that for every
$I\subseteq\{1,\ldots,N\}$ there exists
$g_I\in M_0^\Omega$ satisfying
\begin{equation}
\mathbb P\llbracket\varphi(f_i;g_I)\rrbracket\leq r
\quad\text{if }i\in I,
\qquad
\mathbb P\llbracket\varphi(f_i;g_I)\rrbracket\geq s
\quad\text{if }i\notin I.
\tag{$\dagger$}
\end{equation}

Apply Lemma~\ref{lem:equal-weight-approx} to the finite family
\[
\{g_I:I\subseteq\{1,\ldots,N\}\}
\]
with approximation error $\eta$. Since the dyadic partitions in
that lemma are the fixed partitions $\mathcal B_{2^k}$ used in
the definition of $Q_{2^k}$, there exist a single $k\geq 0$
and elements $g_I'\in M_0^\Omega$, for all
$I\subseteq\{1,\ldots,N\}$, such that every $g_I'$ is constant
on the cells of the same partition
\[
\mathcal B_{2^k}
=
(B_{2^k,1},\ldots,B_{2^k,2^k})
\]
and
\begin{equation}
d_K(g_I,g_I')<\eta
\qquad
\text{for every }I\subseteq\{1,\ldots,N\}.
\tag{1}
\end{equation}

Write
\[
g_I'=\sum_{j=1}^{2^k}b_{I,j}^{B_{2^k,j}},
\qquad b_{I,j}\in M.
\]
For each $I\subseteq\{1,\ldots,N\}$, define
\[
q_I
=
\sum_{j=1}^{2^k}\phi_{b_{I,j},B_{2^k,j}}.
\]
By the definition of $Q_{2^k}$, we have
\begin{equation}
q_I\in Q_{2^k}.
\tag{2}
\end{equation}
Moreover, for every $i\leq N$,
\begin{equation}
\begin{aligned}
q_I(f_i)
&=
\sum_{j=1}^{2^k}
\mathbb P\bigl(
\llbracket\varphi(f_i;b_{I,j}^{\Omega})\rrbracket
\cap B_{2^k,j}
\bigr)\\
&=
\mathbb P\llbracket\varphi(f_i;g_I')\rrbracket.
\end{aligned}
\tag{3}
\end{equation}

By Lemma~\ref{lem:equal-weight-approx}, for every $i\leq N$
and every $I\subseteq\{1,\ldots,N\}$,
\[
\left|
\mathbb P\llbracket\varphi(f_i;g_I)\rrbracket
-
\mathbb P\llbracket\varphi(f_i;g_I')\rrbracket
\right|
<\eta.
\]
Together with $(\dagger)$ and (3), this yields
\begin{equation}
\begin{cases}
q_I(f_i)\leq r+\eta,&\text{if }i\in I,\\[1mm]
q_I(f_i)\geq s-\eta,&\text{if }i\notin I.
\end{cases}
\tag{$\ddagger$}
\end{equation}
In particular,
\begin{equation}
(s-\eta)-(r+\eta)
=\varepsilon-2\eta
=\frac{3\varepsilon}{4}>0.
\tag{4}
\end{equation}

We now show that all the functions $q_I$ belong to the fixed
class $\widehat Q$ from the definition preceding
Lemma~\ref{R-bound}. Recall that
\[
\widehat Q=\bigcup_{\ell\geq1}Q_{\ell!}.
\]
Since $2^k$ divides $\ell!$ for all sufficiently large $\ell$,
the refinement property of the partitions implies
\[
Q_{2^k}\subseteq Q_{\ell!}
\]
for all sufficiently large $\ell$. Indeed, each cell of
$\mathcal B_{2^k}$ is a union of cells of $\mathcal B_{\ell!}$,
and each parameter in the corresponding summand can be kept
constant on all these smaller cells. It follows from (2) that
\begin{equation}
q_I\in\widehat Q
\qquad
\text{for every }I\subseteq\{1,\ldots,N\}.
\tag{5}
\end{equation}

The classes $Q_{\ell!}$ form an increasing sequence. Hence,
by Lemma~\ref{R-bound-2} and Lemma~\ref{R-bound}, for every
$n\geq1$,
\begin{equation}
\begin{aligned}
r_{\widehat Q}(n)
&=\sup_{\ell\geq1}r_{Q_{\ell!}}(n)\\
&\leq\sup_{\ell\geq1}r_Q(n)\\
&=r_Q(n).
\end{aligned}
\tag{6}
\end{equation}
By Theorem~\ref{main}, the class $Q$ has finite VC dimension.
Consequently,
\[
\frac{r_Q(n)}{n}\longrightarrow0
\qquad\text{as }n\to\infty.
\]
It follows from (6) that
\[
\frac{r_{\widehat Q}(n)}{n}\longrightarrow0
\qquad\text{as }n\to\infty.
\]
By Fact~\ref{vc-rademacher}, $\widehat Q$ is therefore a VC
class.

On the other hand, (5) and $(\ddagger)$ show that
$q_I\in\widehat Q$, for each
$I\subseteq\{1,\ldots,N\}$, witnesses
$(r+\eta,s-\eta)$-shattering of
$f_1,\ldots,f_N$. By (4), the gap between these thresholds is
$3\varepsilon/4>0$. Since $\widehat Q$ has finite VC dimension,
the size of a sample that it can
$(r+\eta,s-\eta)$-shatter is bounded. This contradicts our
assumption that such samples exist for arbitrarily large $N$.

We have proved that there is a finite bound
$N=N_{\varphi,r,s}$ as asserted.

The quantitative assertion follows from the quantitative bounds
in Remark~\ref{rem:uniform-rademacher} and Fact~\ref{Ramsey}.
Indeed, by Theorem~\ref{main} and Fact~\ref{Ramsey}, the
scale-sensitive VC dimensions of $Q$ admit bounds depending only
on the VC dimension of $\varphi(x;y)$, the prescribed scale, and
the corresponding Ramsey bound. Remark~\ref{rem:uniform-rademacher}
therefore gives a uniform quantitative bound on
$r_Q(n)/n$. By (6), the same bound applies to
$r_{\widehat Q}(n)/n$. Applying Fact~\ref{rad-shattering-lower}
at the threshold gap $3\varepsilon/4$ gives the required
quantitative bound. Since this gap depends only on
$\varepsilon=s-r$, the resulting bound depends only on the
VC dimension of $\varphi(x;y)$, $\varepsilon$, and the
corresponding Ramsey bound.
\end{proof}

\medskip
Now we are ready to achieve our goal for complete theories and then use it to prove the arbitrary case as well (i.e. Theorem~\ref{thm:arbitrary}).
\begin{Corollary}\label{Corollary}
	For a complete theory \(T\), if $T$ has $NIP$,  then so does $T^R$.
\end{Corollary}
\begin{proof}
	This follows from Theorem~\ref{main} and elimination of quantifiers  in $T^R$ (\cite[Theorem~2.9]{BK}) as follows:
First, by Theorem~\ref{main},  Proposition~\ref{translate} and 
Proposition \ref{independence}
(and the subsequent explanations), for any formula $\phi(x;y)\in L$, the  atomic formula $\mathbb P\llbracket\phi(x;y)\rrbracket$  has finite $VC$-dimension, as witnessed by elements of the (pre-)model $M_0^{\Omega}$.
Notice that the statement ``for $r<s\leq 1$, $\mathbb P\llbracket\phi(x;y)\rrbracket$ has finite $VC$-dimension $\leq N_{\phi,r,s}$" is {\em approximately} expressible\footnote{There is a bit of technicality here in continuous logic: we actually need to present a sequence of statements that together approximate (express) the desired expression. Of course, with Lemma~\ref{lem:uniform-shattering-sentence} we can achieve the goal even with a single sentence.} in $L^R$, and so holds in some/any (pre-)model of $T^R$. See also Remark~\ref{remark-com}. Therefore, ${\mathbb P}\llbracket\phi(x;y)\rrbracket$ has $NIP$.

On the other hand,
  it is easy to see  that if formulas $\varphi, \psi\in L^R$ have $NIP$, then $\neg\varphi$, $\frac{1}{2}\varphi$ and $\varphi\dotdiv\psi$ have $NIP$ as well. Therefore, every quantifier-free formula has $NIP$. The final point is that the uniform limit  of a sequence of $NIP$ formulas has $NIP$.\footnote{In fact, continuous model theorists expand  formulas to uniform limits of formulas in the language, and when they refer to a formula, they mean a definable predicate that is a uniform limit of formulas.} Therefore, by the elimination of quantifiers, $T^R$  has $NIP$.
\end{proof}

The following lemma is a simple technical discussion in continuous logic that is used in the main theorem (i.e., Theorem~\ref{thm:arbitrary}) and is included for the sake of clarity and completeness of the discussion.
\begin{Lemma} 
\label{lem:uniform-shattering-sentence}
Assume $T$ is NIP, let \(\varphi(x;y)\) be an \(L\)-formula, let
\(0\leq r<s\leq 1\), and put \(\varepsilon=s-r\).
Write
\[
\psi(x;y)=\mathbb P\llbracket\varphi(x;y)\rrbracket.
\]
Then the assertion that \(\psi\) has bounded
\((r,s)\)-shattering dimension can be expressed by a single
continuous sentence, with a uniform positive margin that does
not depend on the choice of a complete extension of \(T\).
\end{Lemma}

\begin{proof}
For \(u\in[0,1]\), write \([u]_+=\max\{u,0\}\).
For \(N\geq 1\) and a tuple \(\bar x=(x_1,\ldots,x_N)\), define
\[
\begin{split}
\Delta_{N}^{r,s}(\bar x)
:={}&
\max_{I\subseteq\{1,\ldots,N\}}
\inf_y
\max\Bigl\{
  \max_{i\in I}[\psi(x_i;y)-r]_+,\,
\\[-2mm]
&\hspace{42mm}
  \max_{i\notin I}[s-\psi(x_i;y)]_+
\Bigr\},
\end{split}
\]
where the maximum over an empty set is interpreted as \(0\).
Thus \(\Delta_{N}^{r,s}(\bar x)\) measures the largest, over all
subsets \(I\subseteq\{1,\ldots,N\}\), of the error in realizing
the corresponding shattering pattern. In particular, if
\(\bar x\) is \((r,s)\)-shattered, then
\(\Delta_{N}^{r,s}(\bar x)=0\).

The quantity
\[
\Sigma_{N}^{r,s}
:=
\inf_{x_1,\ldots,x_N}\Delta_{N}^{r,s}(x_1,\ldots,x_N)
\]
is the value of a continuous sentence. The crucial point is
that the absence of exact \((r,s)\)-shattering alone does not
necessarily imply that \(\Sigma_{N}^{r,s}>0\): indeed, a sequence
of tuples may have \(\Delta_{N}^{r,s}\) tending to \(0\) without
ever attaining \(0\). To obtain a positive lower bound, we
therefore use slightly stricter thresholds.

Set
\[
r'=r+\frac{\varepsilon}{4},
\qquad
s'=s-\frac{\varepsilon}{4}.
\]
Then \(r'<s'\) and \(s'-r'=\varepsilon/2\).
By Proposition~\ref{independence}, there is an integer \(N\),
depending only on the uniform data in that proposition for the
gap \(\varepsilon/2\), such that no simple randomization of a
model of any complete extension \(T_1\supseteq T\) can
\((r',s')\)-shatter \(N\) points. Indeed, if
\(d=\operatorname{VC}_T(\varphi)\), then
\[
\operatorname{VC}_{T_1}(\varphi)\leq d
\]
for every such \(T_1\), so the same bound from
Proposition~\ref{independence} works for all complete extensions.
The auxiliary approximation and finite Ramsey bounds occurring
in that proposition are also chosen uniformly.

We claim that, in every such simple randomization,
\[
\Sigma_{N}^{r,s}\geq \frac{\varepsilon}{4}.
\]
To see this, suppose instead that
\(\Delta_{N}^{r,s}(\bar x)<\varepsilon/4\) for some tuple
\(\bar x\). For every \(I\subseteq\{1,\ldots,N\}\), the definition
of \(\Delta_{N}^{r,s}\) then provides a parameter \(y_I\) such
that
\[
\begin{aligned}
\psi(x_i;y_I)&<r+\frac{\varepsilon}{4}=r'
&&\text{for every }i\in I,\\
\psi(x_i;y_I)&>s-\frac{\varepsilon}{4}=s'
&&\text{for every }i\notin I.
\end{aligned}
\]
Consequently, \(\bar x\) is \((r',s')\)-shattered, a contradiction.
It follows that
\(\Delta_{N}^{r,s}(\bar x)\geq\varepsilon/4\) for every
\(\bar x\), and hence the claimed lower bound holds.

Because this is a uniform positive lower bound, the sentence
\[
\Sigma_{N}^{r,s}\geq\frac{\varepsilon}{4}
\]
holds not only in the simple randomizations under consideration
but also in their metric completions. Indeed, \(\Delta_{N}^{r,s}\)
is continuous as a function of \(\bar x\), since \(\inf_y\) is
taken over the compact type space and \(\psi\) is continuous
; consequently, the
inequality \(\Delta_{N}^{r,s}(\bar x)\geq\varepsilon/4\), which
holds on the dense subset of simple randomizations, extends by
continuity to the full metric completion. In particular, the
sentence holds in \((T_1)^R\) for every complete extension
\(T_1\supseteq T\).
\end{proof}

We now derive the preservation of NIP for arbitrary
theories from the uniform complete-theory bound
of Proposition~\ref{independence}. 
We emphasize that the auxiliary theory constructed
below need not be an extension of $T$, nor need it
have the same language. This causes no difficulty:
the quantitative bound in
Proposition~\ref{independence} is uniform over all
complete theories and formulas of bounded VC dimension.
Thus it applies to any auxiliary complete theory
whose relevant relation has the required VC bound.

\begin{Lemma}[Finite-pattern transfer]
\label{lem:finite-pattern-transfer}
Let $T$ be a first-order theory and let
$\varphi(x;y)$ satisfy
$d=\operatorname{VC}_T(\varphi)<\infty$.
For every $0\leq r<s\leq1$, the predicate
\[
\psi(x;y)=\mathbb P\llbracket\varphi(x;y)\rrbracket
\]
cannot $(r,s)$-shatter more than
$N(d+1,s-r)$ points in any model of $T^R$,
where $N$ is the uniform bound from
Proposition~\ref{independence}.
\end{Lemma}

\begin{proof}
Put $N=N(d+1,s-r)$, increasing it to at least
$2$ if necessary. Suppose that in some
$\mathcal N\models T^R$, elements
$f_1,\ldots,f_N$ are $(r,s)$-shattered
by parameters $g_I$, $I\subseteq[N]$.

Let $C_1,\ldots,C_m$ be the positive-measure
atoms of the Boolean algebra generated by
\[
E_{i,I}=\llbracket\varphi(f_i;g_I)\rrbracket,
\]
and put $p_t=\mathbb P(C_t)$.
Define $e_t(i,I)\in\{0,1\}$ by
$C_t\leq E_{i,I}$ if and only if $e_t(i,I)=1$.

Since $p_t>0$, the Boolean pattern $e_t$
is consistent with $T$, by the transfer axioms.
Thus there exist $M_t\models T$ and tuples
$a_i^t,b_I^t$ such that
\[
\varphi(a_i^t;b_I^t)=e_t(i,I).
\]

Construct a finite structure $\mathcal A$
with distinct elements $u_{t,i},v_{t,I}$
and a binary relation $R$ defined by
\[
R(u_{t,i};v_{q,I})
\iff t=q\ \text{and}\ e_t(i,I)=1,
\]
with $R$ false on all remaining pairs.
Let $S=\operatorname{Th}(\mathcal A)$.
Then $S$ is complete.

We claim that $\operatorname{VC}_S(R)\leq d+1$.
Indeed, if $R$ shatters a set $D$ of
$k\geq2$ elements, realizing the full
subset $D$ forces all its elements
to lie in one block, say $t$.
Fix $u_0\in D$. The patterns
$J\cup\{u_0\}$, for
$J\subseteq D\setminus\{u_0\}$,
must be realized within the same block.
Hence the corresponding elements
$a_i^t$ are $(k-1)$-shattered by
$\varphi$ in $M_t\models T$.
Thus $k-1\leq d$, proving the claim.

Now partition an atomless probability space
into measurable sets $D_t$ of measures $p_t$.
In the simple randomization of $\mathcal A$,
define $F_i=u_{t,i}$ and $G_I=v_{t,I}$
on $D_t$. Then
\[
\begin{aligned}
\mathbb P\llbracket R(F_i;G_I)\rrbracket
&=\sum_{t=1}^m p_t e_t(i,I)\\
&=\mathbb P\llbracket
\varphi(f_i;g_I)\rrbracket.
\end{aligned}
\]
Therefore the $F_i$ are $(r,s)$-shattered
in a simple randomization of a model
of the complete theory $S$, although
$\operatorname{VC}_S(R)\leq d+1$.
This contradicts the uniform bound of
Proposition~\ref{independence}.
\end{proof}

\begin{Theorem}\label{thm:arbitrary}
Let $T$ be an arbitrary NIP theory
in the standard randomization setting.
Then $T^R$ is NIP.
\end{Theorem}

\begin{proof}
Fix $\varphi(x;y)$ and $0\leq r<s\leq1$.
Put $\varepsilon=s-r$ and
\[
r'=r+\varepsilon/4,\qquad
s'=s-\varepsilon/4.
\]
Since $T$ is NIP,
$d=\operatorname{VC}_T(\varphi)<\infty$.

By Lemma~\ref{lem:finite-pattern-transfer},
there is $N$ such that
$\psi=\mathbb P\llbracket\varphi\rrbracket$
cannot $(r',s')$-shatter $N$ points
in any model of $T^R$.

The margin argument of
Lemma~\ref{lem:uniform-shattering-sentence}
therefore gives
\[
T^R\models
\Sigma_N^{r,s}\geq\varepsilon/4.
\]
Indeed, any tuple with
$\Delta_N^{r,s}<\varepsilon/4$
would be $(r',s')$-shattered.

Hence every predicate
$\mathbb P\llbracket\varphi(x;y)\rrbracket$
is NIP. By strong quantifier elimination
for randomizations \cite[Theorem~2.9]{BK},
event coding, and preservation of NIP
under continuous connectives and
uniform limits, $T^R$ is NIP.
\end{proof}

\noindent
\subsection{Preservation of  Stability}
The reader might reasonably expect that the method presented above could also be applied in proving the preservation of stability. 
In fact, the above proof can be adapted to yield a proof for the stable case as well, but using the
notions of Littlestone dimension from learning theory—which correspond to the Shelah $2$-rank
in model theory—and sequential Rademacher mean width 
instead of VC-dimension and Rademacher mean width  used in the NIP case.


\bigskip
 Next, we recall the notion of stability in continuous logic.
\begin{Definition}\label{stable def}
\em{
	Let $T_c$ be a complete continuous theory  and $\varphi(x;y)$ a partitioned formula.
	
	\noindent
	(i) We say that $\varphi(x;y)$ has the  $OP$ (order property) if there are  $r<s\leq 1$ and  sequences $(a_i:i<\omega)$ and $(b_j:j<\omega)$ of parameters (in the monster model) such that $\varphi(a_i;b_j)\leq r$ if $i<j$ and $\varphi(a_i;b_j)\geq s$ otherwise. We say that $\varphi(x;y)$ is stable  if it has no $OP$.
	
	\noindent
	(ii) We say that $T_c$ has the $OP$ (or is unstable) if there is some formula $\varphi(x;y)$ with $OP$. $T_c$ is stable if it has no $OP$.
	
	\noindent

	\noindent
	Notice that by setting $r=0$ and $s=1$ in Definition~\ref{stable def}, we obtain exactly the definition of stable/unstable in classical logic.
}
\end{Definition}

\begin{Remark} \label{remark-com-2}
	\em{By compactness, $T_c$  has  $OP$ if there are a formula $\varphi(x;y)$ and $r<s\leq 1$ such that for any $n\in\mathbb N$ and each (pre-)model $M$, there are $a_1,\ldots,a_n, b_1,\ldots,b_n\in M$ such that $\varphi(a_i;b_j)\leq r$ if $i<j$, and $\varphi(a_i,b_j)\geq s$ otherwise.}
\end{Remark}

Let $\mu(x)={\rm Ave}(a_1,\ldots,a_k)$ and $\nu(y)={\rm Ave}(b_1,\ldots,b_k)$ be average measures. In the following, we write $$\mu\otimes\nu(\phi(x;y)):=\frac{1}{k}\sum_{j\leq k}\mu(\phi(x;b_j))=\frac{1}{k}\sum_{i\leq k}\nu(\phi(a_i;y))=\frac{1}{k}\sum_{i\leq k}\frac{1}{k}\sum_{j\leq k}\phi(a_i;b_j).$$

\begin{Theorem}
\label{main-2-2}
Let \(T\) be a stable theory and \(M\models T\).
For every \(L\)-formula \(\phi(x;y)\) and every
\(0\leq r<s\leq 1\), there is a natural number \(N\)
such that there are no sequences
\[
(\mu_1,\ldots,\mu_N)
\quad\text{and}\quad
(\nu_1,\ldots,\nu_N)
\]
of average measures of elements in \(M\) satisfying
\[
\mu_i\otimes\nu_j(\phi(x;y))\leq r
\quad\text{if }i<j,
\]
and
\[
\mu_i\otimes\nu_j(\phi(x;y))\geq s
\quad\text{if }i\geq j.
\]
\end{Theorem}

\begin{proof}
Suppose, towards a contradiction, that the conclusion fails.
Then for arbitrarily large \(N\) there are average measures
\[
\mu_1,\ldots,\mu_N
\quad\text{and}\quad
\nu_1,\ldots,\nu_N
\]
such that
\[
\mu_i\otimes\nu_j(\phi(x;y))\leq r
\quad\text{if }i<j,
\]
and
\[
\mu_i\otimes\nu_j(\phi(x;y))\geq s
\quad\text{if }i\geq j.
\]

Put
\[
\delta:=\frac{s-r}{8}.
\]
Since \(T\) is stable, it is NIP. Applying
Fact~\ref{fact_genstable} to \(\phi(x;y)\) and to the
opposite formula \(\phi^{\mathrm{opp}}(y;x)\), we may
choose fixed lengths for approximating both families of
measures, independently of \(N\). By repeating entries
in the approximating tuples, we may assume that the two
lengths are the same integer \(k\).

Thus there are average measures
\[
\widehat\mu_i
=\operatorname{Ave}(a_{i,1},\ldots,a_{i,k}),
\qquad
\widehat\nu_j
=\operatorname{Ave}(b_{j,1},\ldots,b_{j,k})
\]
such that
\[
\sup_{b\in M}
\left|
\mu_i(\phi(x;b))
-\widehat\mu_i(\phi(x;b))
\right|
\leq\delta
\]
and
\[
\sup_{a\in M}
\left|
\nu_j(\phi(a;y))
-\widehat\nu_j(\phi(a;y))
\right|
\leq\delta.
\]

Replacing the first measure changes the product average
by at most \(\delta\), and replacing the second changes
it by at most another \(\delta\). Hence
\[
\left|
\mu_i\otimes\nu_j(\phi(x;y))
-\widehat\mu_i\otimes\widehat\nu_j(\phi(x;y))
\right|
\leq 2\delta.
\]
Consequently,
\[
\widehat\mu_i\otimes\widehat\nu_j(\phi(x;y))
\leq r+2\delta
\quad\text{if }i<j,
\]
whereas
\[
\widehat\mu_i\otimes\widehat\nu_j(\phi(x;y))
\geq s-2\delta
\quad\text{if }i\geq j.
\]
Our choice of \(\delta\) ensures that
\[
r+2\delta<s-2\delta.
\]

Let \(m\geq 1\) be arbitrary. Since the assumed
configurations exist for arbitrarily large \(N\), we may
choose \(N\) sufficiently large to apply
Fact~\ref{fact:cross-blocks} to the blocks
\[
\bar a_i=(a_{i,p})_{p=1}^{k},
\qquad
\bar b_i=(b_{i,q})_{q=1}^{k},
\]
with target length \(2m\).

We obtain indices
\[
i_1<\cdots<i_{2m}
\]
and matrices \(U,V\in\{0,1\}^{k\times k}\) such that,
for every \(u<v\), \(p\leq k\), and \(q\leq k\),
\[
\phi(a_{i_u,p};b_{i_v,q})=U_{pq},
\qquad
\phi(a_{i_v,p};b_{i_u,q})=V_{pq}.
\]

By the definition of the average measures and the
homogeneity of the cross-block patterns,
\[
\frac{1}{k^2}\sum_{p,q=1}^{k}U_{pq}
=
\widehat\mu_{i_u}\otimes\widehat\nu_{i_v}
(\phi(x;y))
\leq r+2\delta
\]
for every \(u<v\), while
\[
\frac{1}{k^2}\sum_{p,q=1}^{k}V_{pq}
=
\widehat\mu_{i_v}\otimes\widehat\nu_{i_u}
(\phi(x;y))
\geq s-2\delta.
\]
It follows that
\[
\sum_{p,q=1}^{k}U_{pq}
<
\sum_{p,q=1}^{k}V_{pq}.
\]
Since all entries of \(U\) and \(V\) belong to
\(\{0,1\}\), there are \(p,q\leq k\) such that
\[
U_{pq}=0,\qquad V_{pq}=1.
\]

Define
\[
c_t:=a_{i_{2t},p},
\qquad
d_u:=b_{i_{2u-1},q}
\qquad (1\leq t,u\leq m).
\]
If \(t<u\), then \(2t<2u-1\), so the forward
cross-block pattern gives
\[
\neg\phi(c_t;d_u).
\]
If \(t\geq u\), then \(2t>2u-1\), so the reverse
cross-block pattern gives
\[
\phi(c_t;d_u).
\]
Therefore,
\[
\phi(c_t;d_u)
\quad\Longleftrightarrow\quad t\geq u.
\]

Since \(m\) was arbitrary, \(\phi(x;y)\) has the order
property. This contradicts the stability of \(T\).
Hence the required natural number \(N\) exists.
\end{proof}

\begin{Remark} 
\em{ Similar to the $NIP$ case (cf. Remark~\ref{Remark quantitative}),
 with the condition that at least one of the variables \(x\) or \(y\) is interpreted only by elements of the classical models,
 Theorem~\ref{main-2-2} provides a quantitative bound for not having order property for formula ${\mathbb P}\llbracket\phi(x;y)\rrbracket$.}
\end{Remark}

\subsection*{Sequential Rademacher Mean Width and Littlestone Dimension}
The preservation of stability under randomization was established for classical theories in Theorem~5.14 of \cite{BK} and for continuous theories in Theorem~4.9 of \cite{BenYaacov2013}. In particular, if \(T\) is stable, then so is its randomization \(T^R\).
Using ideas from the present paper, 
this follows from Theorem~\ref{main-2-2} (instead of Theorem~\ref{main}), a stable version of Proposition ~\ref{independence}, Proposition \ref{transfer} and elimination of quantifiers, similar to the proof of Corollary~\ref{Corollary}. 
Of course, to prove a proposition analogous to Proposition~\ref{independence} in the stable case, we need the
notion of Littlestone dimension (which corresponds to Shelah $2$-rank) and Sequential Rademacher mean width instead of VC-dimension
and Rademacher mean width used in the NIP case, which we address in this paper.  


\begin{Definition}[Sequential Rademacher mean width]
Let \(X\) be a set and let \(\mathcal F\subseteq [0,1]^X\).
For \(n\geq 1\), an \(X\)-valued tree of depth \(n\) is a family
\[
\mathbf x=
\bigl(x_t(\epsilon_1,\ldots,\epsilon_{t-1})\bigr)_{t=1}^n,
\]
where
\[
x_t:\{-1,+1\}^{t-1}\longrightarrow X.
\]
Define the sequential Rademacher mean width of \(\mathcal F\) by
\[
\mathfrak R_n^{\mathrm{seq}}(\mathcal F)
=
\sup_{\mathbf x}
\mathbb E_{\epsilon}
\left[
\sup_{f\in\mathcal F}
\sum_{t=1}^n
\epsilon_t
f\bigl(x_t(\epsilon_{<t})\bigr)
\right],
\]
where \(\epsilon_1,\ldots,\epsilon_n\) are independent
Rademacher random variables and
\(\epsilon_{<t}=(\epsilon_1,\ldots,\epsilon_{t-1})\).
\end{Definition}

We use the following qualitative consequence of \cite[Theorem~1]{RakhlinSridharanTewari2015Sequential}. \begin{Fact}\label{fact:seq-rad-littlestone} If \(\mathcal H\subseteq\{0,1\}^X\) has finite Littlestone dimension, then \[ \lim_{n\to\infty} \frac{\mathfrak R_n^{\mathrm{seq}}(\mathcal H)}{n}=0. \] \end{Fact}

\begin{proof}
For every \(0<\epsilon\leq 1/2\), the binary-valued
nature of \(\mathcal H\) implies
\[
\operatorname{sfat}_{\epsilon}(\mathcal H)
\leq \operatorname{Ldim}(\mathcal H)<\infty.
\]
Hence \(\mathcal H\) has finite sequential
fat-shattering dimension at every positive scale.
The asserted convergence now follows from
\cite[Theorem~1]{RakhlinSridharanTewari2015Sequential}.
\end{proof}

Fix an \(L\)-formula \(\varphi(x;y)\), a model \(M\models T\),
and a simple randomization \(M_0^\Omega\) of \(M\). Put
\[
X=M_0^\Omega
\qquad\text{and}\qquad
Q=\{\varphi_b:b\in M^{|y|}\},
\]
where
\[
\varphi_b(f)
=
\mathbb P\llbracket\varphi(f;b^\Omega)\rrbracket.
\]

We first show that
\[
\lim_{n\to\infty}
\frac{\mathfrak R_n^{\mathrm{seq}}(Q)}{n}=0.
\tag{1}\label{eq:seq-Q}
\]

Fix \(\eta>0\). For each \(f\in M_0^\Omega\), let
\(\mu_f\) be the finitely supported Keisler measure defined by
\[
\mu_f(\varphi(x;b))
=
\mathbb P\llbracket\varphi(f;b^\Omega)\rrbracket.
\]
Since \(\mu_f\) is a finite convex combination of realized types,
it is generically stable. By Fact~\ref{fact_genstable}, there is
an integer \(k\), independent of \(f\), such that for every
\(f\in M_0^\Omega\) there is a tuple
\(\bar a_f=(a_{f,1},\ldots,a_{f,k})\in M^k\) satisfying
\[
\sup_{b\in M^{|y|}}
\left|
\mathbb P\llbracket\varphi(f;b^\Omega)\rrbracket
-
\frac1k\sum_{i=1}^k
\mathbf 1_{\varphi(a_{f,i};b)}
\right|
\leq\eta.
\tag{2}\label{eq:seq-approx}
\]

For this fixed \(k\), let \(X_k=M^k\) and define
\[
F_b(\bar a)
=
\frac1k\sum_{i=1}^k
\mathbf 1_{\varphi(a_i;b)}
\qquad
(\bar a\in M^k,\ b\in M^{|y|}).
\]
Write
\[
\mathcal F_k=\{F_b:b\in M^{|y|}\}.
\]

For each \(1\leq \ell\leq k\), define the binary class
\[
\mathcal H_{k,\ell}
=
\{h_{b,\ell}:b\in M^{|y|}\},
\]
where
\[
h_{b,\ell}(\bar a)=1
\quad\Longleftrightarrow\quad
\sum_{i=1}^k
\mathbf 1_{\varphi(a_i;b)}\geq\ell.
\]
Notice that
\[
F_b(\bar a)
=
\frac1k\sum_{\ell=1}^k h_{b,\ell}(\bar a).
\tag{3}\label{eq:level-decomposition}
\]

We claim that every \(\mathcal H_{k,\ell}\) has finite Littlestone
dimension. Indeed, the condition
\[
\sum_{i=1}^k\mathbf 1_{\varphi(a_i;b)}\geq\ell
\]
is expressed by an \(L\)-formula
\(\theta_{k,\ell}(\bar x;y)\). If this formula had arbitrarily
long order patterns, setting
\[
\mu_i=\operatorname{Ave}(a_{i,1},\ldots,a_{i,k})
\quad\text{and}\quad
\nu_j=\operatorname{Ave}(b_j)
\]
would give
\[
\mu_i\otimes\nu_j(\varphi(x;y))
\leq \frac{\ell-1}{k}
\quad\text{if }i<j,
\]
and
\[
\mu_i\otimes\nu_j(\varphi(x;y))
\geq \frac{\ell}{k}
\quad\text{if }i\geq j.
\]
This contradicts Theorem~\ref{main-2-2}, applied with
\(r=(\ell-1)/k\) and \(s=\ell/k\).
Thus \(\theta_{k,\ell}\) is stable, and its Littlestone dimension
is finite.

By Fact~\ref{fact:seq-rad-littlestone}, for each
\(\ell\leq k\),
\[
\mathfrak R_n^{\mathrm{seq}}(\mathcal H_{k,\ell})=o(n).
\]
Using \eqref{eq:level-decomposition} and the subadditivity of
sequential Rademacher mean width, we obtain
\[
\mathfrak R_n^{\mathrm{seq}}(\mathcal F_k)
\leq
\frac1k\sum_{\ell=1}^k
\mathfrak R_n^{\mathrm{seq}}(\mathcal H_{k,\ell})
=o(n).
\tag{4}\label{eq:seq-Fk}
\]

Now consider an arbitrary \(X\)-valued tree
\(\mathbf f=(f_t(\epsilon_{<t}))_{t=1}^n\).
At every node, choose the tuple
\(\bar a_{f_t(\epsilon_{<t})}\) supplied by
\eqref{eq:seq-approx}. These tuples form an \(X_k\)-valued tree.
Consequently, \eqref{eq:seq-approx} gives
\[
\mathfrak R_n^{\mathrm{seq}}(Q)
\leq
n\eta+\mathfrak R_n^{\mathrm{seq}}(\mathcal F_k).
\]
Dividing by \(n\), using \eqref{eq:seq-Fk}, and then letting
\(\eta\) tend to zero proves \eqref{eq:seq-Q}.

We next establish the sequential analogue of the Rademacher
mean-width bound used in the NIP argument. Use the same nested
family of equal-measure partitions
\[
\mathcal B_k=(B_{k,1},\ldots,B_{k,k})
\]
as in the definition of \(Q_k\). Recall that
\[
\varphi_{a,B}(f)
=
\mathbb P\bigl(
\llbracket\varphi(f;a^\Omega)\rrbracket\cap B
\bigr)
\]
and
\[
Q_k=
\left\{
\sum_{j=1}^k\varphi_{a_j,B_{k,j}}:
a_1,\ldots,a_k\in M
\right\}.
\]
Then
\[
\mathfrak R_n^{\mathrm{seq}}(Q_k)
\leq
\mathfrak R_n^{\mathrm{seq}}(Q).
\tag{5}\label{eq:seq-slicing}
\]

To see this, fix an \(X\)-valued tree
\(\mathbf f=(f_t(\epsilon_{<t}))_{t=1}^n\).
For each cell \(B_{k,j}\), let \(\mathbb P_j\) be the
normalized restriction of \(\mathbb P\) to \(B_{k,j}\).
Since the random variables \(f_t(\epsilon_{<t})\) are simple
and the probability space is atomless, for each node there is
a simple random variable \(f'_{t,j}(\epsilon_{<t})\) on
\(\Omega\) having the same distribution under \(\mathbb P\)
as \(f_t(\epsilon_{<t})\) has under \(\mathbb P_j\).
Consequently, for every \(a\in M\),
\[
\varphi_{a,B_{k,j}}
\bigl(f_t(\epsilon_{<t})\bigr)
=
\frac1k
\varphi_a\bigl(f'_{t,j}(\epsilon_{<t})\bigr).
\]
The family \((f'_{t,j})_{t=1}^n\) is again an \(X\)-valued
tree. Since the parameters \(a_1,\ldots,a_k\) in the definition
of \(Q_k\) can be optimized independently, we get
\[
\begin{aligned}
&\mathbb E_\epsilon
\sup_{q\in Q_k}
\sum_{t=1}^n
\epsilon_t q\bigl(f_t(\epsilon_{<t})\bigr)\\
&\qquad\leq
\sum_{j=1}^k
\frac1k\mathfrak R_n^{\mathrm{seq}}(Q)
=
\mathfrak R_n^{\mathrm{seq}}(Q).
\end{aligned}
\]
Taking the supremum over all such trees proves
\eqref{eq:seq-slicing}.

As in the NIP proof, put
\[
\widehat Q=\bigcup_{m\geq1}Q_{m!}.
\]
The classes \(Q_{m!}\) are increasing. It follows from
\eqref{eq:seq-slicing} that
\[
\mathfrak R_n^{\mathrm{seq}}(\widehat Q)
\leq
\mathfrak R_n^{\mathrm{seq}}(Q).
\tag{6}\label{eq:seq-Qhat}
\]
In particular, by \eqref{eq:seq-Q},
\[
\lim_{n\to\infty}
\frac{\mathfrak R_n^{\mathrm{seq}}(\widehat Q)}{n}=0.
\tag{7}\label{eq:seq-Qhat-small}
\]

We can now prove preservation of stability.

\begin{Theorem}\label{thm:stability-randomization-seq}
 For a complete theory $T$, 
if \(T\) is stable, then \(T^R\) is stable.
\end{Theorem}

\begin{proof}
Suppose, towards a contradiction, that \(T^R\) is not stable.
By quantifier elimination for randomizations and the order-property
criterion for stability, there are an \(L\)-formula
\(\varphi(x;y)\) and \(0\leq r<s\leq1\) such that, for every
\(N\), there are
\[
f_1,\ldots,f_N,\ g_1,\ldots,g_N\in M_0^\Omega
\]
with
\[
\mathbb P\llbracket\varphi(f_i;g_j)\rrbracket\leq r
\quad\text{if }i<j,
\tag{8}\label{eq:random-order-low}
\]
and
\[
\mathbb P\llbracket\varphi(f_i;g_j)\rrbracket\geq s
\quad\text{if }i\geq j.
\tag{9}\label{eq:random-order-high}
\]

Let
\[
\eta=\frac{s-r}{8}.
\]
Fix \(N\). By the equal-weight approximation lemma
(Lemma~\ref{lem:equal-weight-approx}), applied to the finite
family \(g_1,\ldots,g_N\), we may choose a common equal-measure
partition
\[
\mathcal B_k=(B_{k,1},\ldots,B_{k,k})
\]
and random variables \(g'_1,\ldots,g'_N\), each constant on
the cells of this partition, such that
\[
\left|
\mathbb P\llbracket\varphi(f_i;g_j)\rrbracket
-
\mathbb P\llbracket\varphi(f_i;g'_j)\rrbracket
\right|<\eta
\]
for every \(i,j\leq N\).

Write
\[
g'_j=\sum_{\ell=1}^k b_{j,\ell}^{B_{k,\ell}},
\qquad b_{j,\ell}\in M,
\]
and put
\[
q_j=\sum_{\ell=1}^k\varphi_{b_{j,\ell},B_{k,\ell}}.
\]
Then \(q_j\in Q_k\subseteq\widehat Q\), and
\[
q_j(f_i)
=
\mathbb P\llbracket\varphi(f_i;g'_j)\rrbracket.
\]
It follows from \eqref{eq:random-order-low} and
\eqref{eq:random-order-high} that
\[
q_j(f_i)\leq r+\eta
\quad\text{if }i<j,
\tag{10}\label{eq:q-low}
\]
and
\[
q_j(f_i)\geq s-\eta
\quad\text{if }i\geq j.
\tag{11}\label{eq:q-high}
\]
Set
\[
r'=r+\eta,\qquad s'=s-\eta.
\]
Then
\[
s'-r'=(s-r)-2\eta>0.
\tag{12}\label{eq:positive-gap}
\]

Let \(d\geq1\). Choose \(N\geq2^d\) and take the corresponding
configuration above. The ladder pattern
\eqref{eq:q-low}--\eqref{eq:q-high} yields a binary tree of
depth \(d\) whose nodes are selected from
\(f_1,\ldots,f_N\) and which is
\((r',s')\)-shattered by \(\widehat Q\).
Indeed, recursively query a median point \(f_i\) among the
remaining indices \(j\): by \eqref{eq:q-low}--\eqref{eq:q-high},
the two outcomes retain respectively the indices \(j>i\)
and \(j\leq i\). Since \(N\geq2^d\), every path of length \(d\)
has at least one remaining witness \(q_j\).

Write
\[
c=\frac{r'+s'}2,
\qquad
\alpha=\frac{s'-r'}2.
\]
For every path
\(\epsilon=(\epsilon_1,\ldots,\epsilon_d)\in\{-1,+1\}^d\),
choose a witness \(q_\epsilon\in\widehat Q\) for the tree,
with the convention that \(\epsilon_t=+1\) corresponds to the
upper threshold and \(\epsilon_t=-1\) to the lower threshold.
Then
\[
\epsilon_t
\bigl(q_\epsilon(x_t(\epsilon_{<t}))-c\bigr)
\geq\alpha
\qquad (1\leq t\leq d).
\]
Therefore,
\[
\sup_{q\in\widehat Q}
\sum_{t=1}^d
\epsilon_tq(x_t(\epsilon_{<t}))
\geq
c\sum_{t=1}^d\epsilon_t+\alpha d.
\]
Taking expectation over \(\epsilon\), and using
\(\mathbb E[\sum_{t=1}^d\epsilon_t]=0\), gives
\[
\mathfrak R_d^{\mathrm{seq}}(\widehat Q)
\geq
\alpha d
=
\frac{s'-r'}2\,d.
\tag{13}\label{eq:seq-lower}
\]

By \eqref{eq:seq-Qhat-small}, for all sufficiently large \(d\),
\[
\frac{\mathfrak R_d^{\mathrm{seq}}(\widehat Q)}d
<
\frac{s'-r'}2.
\]
This contradicts \eqref{eq:seq-lower}. Hence \(T^R\) is stable.
\end{proof}

\begin{Remark}
Although the proof of Theorem~\ref{thm:stability-randomization-seq} works for the preservation of stability for complete theories, similarly to the $NIP$ case (i.e. Proposition~\ref{independence})
one can find quantitative bounds for 
sequential Rademacher mean width and the Littlestone dimension 
  such that exactly the idea of the proof of  Theorem~\ref{thm:arbitrary} yields the preservation of stability for arbitrary theories. Since the aim of the paper is to emphasize the main idea, which works for both the NIP and stable cases, and to avoid emphasizing the technical aspect, we have not studied these quantitative bounds for the sequential Rademacher mean width and the Littlestone dimension in this paper; we will pursue this elsewhere.
\end{Remark}

\bigskip\noindent
\subsection{$NSOP$ and Shelah's Theorem for $T^R$}  Before the discussion, we emphasize that in this section \(T\), and consequently \(T^R\), are \emph{complete} theories in classical first-order logic and continuous logic, respectively. 
The concept of strict order property ($SOP$) has various generalizations in continuous logic, e.g.  \cite[Definition~1.1(i)]{K-sop}.  This notion was further studied in   \cite{K-classification}. The key feature of this property is highlighted within the classical model theory, by a theorem of Shelah which states a first-order complete theory is stable if and only if it is $NIP$ and $NSOP$. In this part of the paper we show that, in the context of randomization, Definition \cite[Definition~1.1(i)]{K-sop} is the weakest condition amongst its counterpart definitions which satisfies the   Shelah's theorem. That is, $T^R$ is stable if and only if it is $NIP$ and $NSOP$.
We first review the notion of $SOP$ of  \cite{K-classification} for continuous logic.

In the following, $\cal U$ and $\cal C$ denote the monster models of $T$ and $T^R$, respectively. Recall that in continuous logic $1$ is ``False" and $0$ is ``True", but this is a minor issue, and for consistency in notation,  we continue to use $\varphi=1$ if $\models\varphi$ and $\varphi=0$ if $\models\neg\varphi$.

\begin{Definition}[\cite{K-sop} Definition~1.1(i), and \cite{K-classification}, Definition~3.3] \label{SOP-definition}
	{\em  Let $T_c$ be a complete continuous theory, $\varphi(x;y)$ a partitioned formula and $\epsilon>0$.
		
		\noindent (i)
		We say that the formula $\varphi(x;y)$  has the \emph{$\epsilon$-strict order property} ($\epsilon$-$SOP$) if there exists a sequence
		$(a_ib_i:i<\omega)$ in   the monster model $\cal C$ of $T_c$ such that for any $b\in\cal C$ and for all $i<j$,
		$$\varphi(b;a_i)\leqslant\varphi(b;a_{i+1}), \ \mbox{ and } \ \varphi(b_j;a_i)+\epsilon\leqslant\varphi(b_{i+1};a_j).$$
		
		\medskip\noindent
		(ii) We say that the formula $\varphi(x;y)$  has the
		{\em strict order property} ($SOP$) if it has $\epsilon$-$SOP$ for
		some $\epsilon>0$.
		
		\medskip\noindent
		(iii) We say that $T_c$ has $SOP$ if there is some formula $\varphi(x;y)$
		with $SOP$. We say that $T_c$ has the $NSOP$ if it does not have the $SOP$.}
\end{Definition}


\begin{Proposition}\label{sop->sop}
If $T$ has $SOP$  then so has $T^R$.
\end{Proposition}
\begin{proof}
Let $\Omega$ be an atomless probability space and ${\cal U}_0^\Omega$ the simple pre-model with underlying probability space $\Omega$ (as in Subsection~2.2).
As $T$ has $SOP$, there are a formula $\phi(x;y)$ and  sequences $(a_i:i<\omega)$ and $(b_j:j<\omega)$ of elements in $\cal U$ such that $\phi(b;a_i)\leq\phi(b;a_{i+1})$ for any $b\in\mathcal U$, and $\models\neg\phi(b_{i+1};a_i)\wedge\phi(b_{i+1};a_{i+1})$ for each $i<\omega$. Therefore,
${\mathbb P}\llbracket\phi(b;a_i) \rrbracket\leq{\mathbb P}\llbracket\phi(b;a_{i+1})\rrbracket$, for any $b\in\cal U$. (Notice that in  the previous sentence  $b$ is the constant map $f_b:\Omega\to\{b\}$ defined by $t\mapsto b$. Similarly for the $a_i$'s.) Also, ${\mathbb P}\llbracket\phi(b_{i+1};a_i)\rrbracket=0$ and ${\mathbb P}\llbracket\phi(b_{i+1};a_{i+1})\rrbracket=1$ for each $i<\omega$.

 It is easy to verify that for any $f\in{\cal U}_0^\Omega$, ${\mathbb P}\llbracket\phi(f;a_i)\rrbracket\leq{\mathbb P}\llbracket\phi(f;a_{i+1})\rrbracket$.
Indeed, we can assume that every $f\in{\cal U}_0^\Omega$ is of the form $\sum_{j\leq k}c_j^{{\cal A}_i}$ where $c_1,\ldots,c_k\in\cal U$ and ${\cal A}_1,\ldots,{\cal A}_k$ is a measurable partition of $\Omega$. Therefore,
\begin{align*}
{\mathbb P}\llbracket\phi(f;a_{i})\rrbracket & ={\mathbb P}\llbracket\phi(\sum_{j\leq k}c_j^{{\cal A}_j};a_i)\rrbracket=\sum_{j\leq k}\mu_0({\cal A}_j)\cdot{\mathbb P}\llbracket\phi(c_j;a_i)\rrbracket \\
& \leq \sum_{j\leq k}
\mu_0({\cal A}_j)\cdot{\mathbb P}\llbracket\phi(c_j;a_{i+1})\rrbracket={\mathbb P}\llbracket\phi(\sum_{j\leq k}c_j^{{\cal A}_j};a_{i+1})\rrbracket  \\ & = {\mathbb P}\llbracket\phi(f;a_{i+1})\rrbracket.
\end{align*}
 As ${\cal U}_0^\Omega$ is metrically dense in ${\cal U}^\Omega$, the same holds for any $f\in{\cal U}^\Omega$.
As ${\cal U}^\Omega$ is an elementary substructure of $\cal C$, we have ${\mathbb P}\llbracket\phi(f; a_i)\rrbracket\leq{\mathbb P}\llbracket\phi(f;a_{i+1})\rrbracket$ for all $f\in\cal C$ and  for any $i<\omega$. Therefore, $T^R$ has $SOP$.
\end{proof}

\begin{Corollary}[Shelah's theorem for randomization theories]\label{continuous-Shelah}
$T^R$ is stable if and only if it is $NIP$ and $NSOP$.
\end{Corollary}
\begin{proof}
Left to right is evident.
For the converse, suppose that $T^R$ has $NIP$ but it is not stable. Clearly, $T$ has $NIP$. 
Since stability is preserved under randomization,
$T$ is not stable.
Now, by Shelah's theorem for classical theories (\cite[Theorem~2.67]{S}), $T$ has $SOP$. By Proposition~\ref{sop->sop}, $T^R$ has $SOP$.
\end{proof}

\begin{Remark}\label{reasons}
\em{
  Despite the fact that, there is no universally accepted definition for $SOP$ in continuous logic, there are some justifications as why  Definition~\ref{SOP-definition} could be a more suitable candidate for this notion in continuous logic. In \cite{Hanson}, Hanson studied some $SOP$-like properties for continuous logic.
 We do not claim that the notion of $NSOP$ of  Definition~\ref{SOP-definition}  is the weakest among all possible definitions of $NSOP$ in continuous logic. Rather, among the definitions considered in \cite{Hanson}, it is the weakest one for which the corresponding stability characterization holds in the context of randomization.
We claim that if $\mathbb P\llbracket\psi(y;z)\rrbracket$ has the $SOP$
in the sense of Definition~\ref{SOP-definition}, then $\psi(y;z)$
also has the $SOP$, and 
by Proposition B.5.5 in \cite{Hanson}, $T$ has $SOP_n$ for all $n$. (Cf. \cite[Def. B.5.4]{Hanson} for the definition of $SOP_n$.) This means that our definition of $SOP$ is strong enough that Corollary~\ref{continuous-Shelah}  does not state anything trivial/tautological.}
\end{Remark}

\section{Sketch of the  Preservation Theorems for Continuous Theory $T$}
In this section, we sketch  how the preservation theorems can be extended to the situation where $T$ itself is a continuous theory. To this end, in the NIP case, we will explain those parts  with more details whenever the proof differs from the discrete case. 

For the reasons that will be explained later (cf. Remark~\ref{remark}), the presentation of the proof sketch  for the continuous case is somewhat different from the discrete one,  although one  could  easily translate the above method into the present context. The key point is that in the continuous case, all the tools, including simple models and the continuous version of Fact~\ref{fact_genstable}, are also available.



For  the rest of this section, we assume  $T$ to be a {\bf continuous} theory and $M$ a model of $T$.
Fix $n,k\in \mathbb N$ and for each $i\leq n$, assume $\bar x_i=(x_{i1},\ldots,x_{ik})$ to be $n$ tuples of  sequence of variables of length $k$. Let    $A=\{r,s\}$ with $r<s\leq 1$, and $\phi(x;y)$ an $L$-formula.
As before ${\rm Ave}(\bar x_i)(\phi(x;y))$ stands for $\frac{1}{k}\sum_{j=1}^k\phi(x_{ij};y)$.

\noindent
A $\phi$-$k$-$n$-$A$-condition $\Phi(\bar x_1,\ldots, \bar x_n)$ is an $L$-condition of the form in

 $\exists y\big(\bigwedge_{i\in I} {\rm Ave}(\bar x_i)(\phi(x;y))\leq r\wedge\bigwedge_{i\notin I} {\rm Ave}(\bar x_i)(\phi(x;y))\geq s\big) $ or

 $\forall y\big(\bigvee_{i\in I} {\rm Ave}(\bar x_i)(\phi(x;y))\leq r\vee\bigvee_{i\notin I} {\rm Ave}(\bar x_i)(\phi(x;y))\geq s\big)$

\noindent
where    $I\subseteq\{1,\ldots,n\}$.

 Notice that every $\phi$-$k$-$n$-$A$-condition is expressible by a (continuous)  first-order formula/condition (which only use $A$). For a tuple $\boldsymbol{b} =(\bar a_1,\ldots,\bar a_n)\in(M^k)^n$,  the $\phi$-$k$-$n$-$A$-type of $\boldsymbol{b}$, denoted by $tp_{\phi,k,n,A}(\boldsymbol{b})$, is the set of all $\phi$-$k$-$n$-$A$-conditions $\Phi(\bar x_1,\ldots,\bar x_n)$ such that $\models\Phi(\boldsymbol{b})$.
\begin{Definition} {\em Fix $n,k\in \mathbb N$ and $A=\{r,s\}$ with $r<s$.
	Let $(\bar a_i:i\leq\omega)$ be a sequence of $k$-tuples of parameters. By the above notation, we say that a sequence $(\bar a_i:i\leq\omega)$ is a $\phi$-$k$-$n$-$A$-indiscernible sequence (over the empty set) if for each $i_1<\cdots<i_n<\omega$ and $j_1<\cdots<j_n<\omega$,
	$${\rm tp}_{\phi,k,n,A}(\bar a_{i_1},\ldots,\bar a_{i_n})={\rm tp}_{\phi,k,n,A}(\bar a_{j_1},\ldots,\bar a_{j_n}).$$ }
\end{Definition}

The following fact follows from the finite Ramsey theorem.
\begin{Fact} \label{Ramsey 2}
	Fix $n,k,m\in\mathbb N$ and $A=\{r,s\}$ with $r<s\leq 1$. Then there is some large number $N\in\mathbb N$
	such that for every sequence $(\bar a_1,\ldots,\bar a_N)$ of $k$-tuples of elements  of  $M$ there exists a $\phi$-$k$-$n$-$A$-indiscernible subsequence  $(\bar a_{i_1},\ldots,\bar a_{i_m})$ of length $m$.
\end{Fact}

Fix $N\in\omega+1$. Let $(\mu_i(x):i<N)$ be a sequence of Keisler measures on variable $x$, $\phi(x;y)$ a formula, and $r<s\leq 1$. We say that $(\mu_i:i< N)$ is shattered by elements in $M$ for $\phi(x;y)$ and $r<s$ if for any $I\subseteq\{i:i<N\}$ there is $b_I\in M$ such that $\mu_i(\phi(x;b_I))\leq r$ if $i\in I$ and $\mu_i(\phi(x;b_I))\geq s$ if $i\notin I$. In this case, sometimes we simply say that the sequence is shattered (by elements in $M$). 
\begin{Lemma} \label{main lemma}
	Let $T$ be an $NIP$ theory and $M$ a model of $T$. Assume that there is a formula $\phi(x;y)$ and $r<s\leq 1$, such that for each natural number $n$, there are average measures $(\mu_{1,n},\ldots,\mu_{n,n})$ (where $\mu_{i,n}$ is of the form ${\rm Ave}(a_1,\ldots,a_k)$ with $a_1,\ldots,a_k\in M$) such that $(\mu_{1,n},\ldots,\mu_{n,n})$ is shattered by elements in $M$ for $\phi(x;y)$ and $r<s$. Then there is an infinite sequence $(\hat\mu_i:i<\omega)$ of pseudo-finite measures 
which is shattered by elements in the monster model for $\phi(x;y)$ and $r<s$.
\end{Lemma}
\begin{proof}[Sketch of the proof]  
By the assumption, using the ultrapower and ultralimit constructions, it is easy to construct a sequence $(\hat\mu_i:i<\omega)$ of pseudo-finite measures 
 with the desired property. Notice that definability of the $\hat\mu_i$'s will be used in the proof. (Indeed, similar to the classical case, definability follows from the fact that every pseudo-finite measure in an $NIP$ theory is generically stable. Cf. Lemma~3.13 of \cite{Anderson2}.)
\end{proof}

\begin{Theorem}\label{average shattered}
	Let $T$ be a complete continuous theory with $NIP$. Then for each model $M$ of $T$, and formula $\phi(x;y)$ and $r<s\leq 1$, there is some natural number $n$, such that there is {\bf  no} average measures $(\mu_{1},\ldots,\mu_{n})$ (where $\mu_{i}$ is of the form ${\rm Ave}(a_1,\ldots,a_k)$ with  $a_1,\ldots,a_k\in M$) such that $(\mu_1,\ldots,\mu_n)$ is shattered by elements in $M$.
\end{Theorem}
\begin{proof}[Sketch of the  proof]
	Assume for a contradiction that there is no such $n$ for formula $\phi(x;y)$ and $r<s$. By Lemma~\ref{main lemma},  then there is an infinite sequence $(\hat\mu_i:i<\omega)$ of pseudo-finite measures 
 which is shattered by elements in the monster model. On the other hand, as $T$ has $NIP$, by Lemma~3.13 of \cite{Anderson2},
  every $\hat\mu_i$ is generically stable (and so definable and finitely satisfiable). Using a suitable translation\footnote{Such a translation is provided in Subsection 4.2 of \cite{K-classification}.}  of the equivalence of (iv)~$\Longleftrightarrow$~(vi) in Theorem~2F of \cite{BFT},  we  can conclude that  there is a non-Borel definable measure in the closure of  $\{\hat\mu_i:i<\omega\}$ in the space of measures. This contradicts the fact that finitely satisfiable (and even invariant) measures in $NIP$ theories are Borel definable (cf. Lemma~3.10 of \cite{Anderson2}).
\end{proof}

In the following  theorem, the notion of shattering by elements in $M_0^{\Omega}$ is essentially the same as that defined before Theorem~\ref{main}.
\begin{Theorem}\label{main 2}
Suppose $T$ is a continuous  $NIP$ theory, $M$ a model of $T$, and $M_0^\Omega$ a simple randomization of $M$. For each $r<s$, there is a natural number $N$ such that there is no $f_1,\ldots,f_N\in M_0^\Omega$ such that for any $I\subseteq\{1,\ldots,N\}$ there is some $g_I\in M_0^\Omega$ with the following property:
$$\mathbb E\llbracket\phi(f_i;g_I)\rrbracket\leq r, \ \ \ \text{ if } \ i\in I, \text{ and }
$$
$$\mathbb E\llbracket\phi(f_i;g_I)\rrbracket\geq s, \ \ \ \text{ if } \ i\notin I. 
 \ \ \ \ \ \ \ $$

\end{Theorem}
\begin{proof}[Sketch of the  proof]
	The proof is essentially similar to Proposition~\ref{independence}. 
Note that Lemmas \ref{R-bound} and \ref{R-bound-2} also hold in the continuous case.
The difference is that we use Theorem~\ref{average shattered}, Fact~\ref{Ramsey 2}, and a continuous version of Lemma~\ref{lemma}.\footnote{An adaptation of Fact~2.3 of \cite{K-classification} gives such a lemma.} (Notice that one can use  Corollary~3.7 of \cite{Anderson2} instead of Fact~\ref{fact_genstable} above.)
\end{proof}
\begin{Corollary}
	If the continuous theory $T$ is $NIP$ then so is $T^R$.
\end{Corollary}
\begin{proof}
	Similar to Corollary~\ref{Corollary}, this follows from Theorem~\ref{main 2} and elimination of quantifiers  in $T^R$. 
\end{proof}

We conclude the paper with the following remark.
\begin{Remark}\label{remark}
{\em 	(i) The reader can easily verify that a direct translation of the proof presented for classical logic to the continuous case is entirely possible, and in this case, Lemma \ref{main lemma} and Theorem \ref{average shattered} (and so Theorem~2F of \cite{BFT}) are not needed. Nevertheless, we believe that Lemma~\ref{main lemma} and Theorem~\ref{average shattered}  are interesting in their own right and may inspire other results elsewhere. Additionally, establishing a connection between analysis and logic is intriguing and valuable.
\newline
(ii) 
As we mentioned regarding the preservation of randomization from discrete models, the notion of
Littlestone dimension must replace Rademacher complexity and VC-dimension. Assuming these
tools are available, 
for the preservation of stability in continuous theories, one can use Grothendieck's theorem instead of Theorem~2F of \cite{BFT}. (See Fact~2.6 of \cite{K-classification} for the translation of Grothendieck's theorem into the language of logic).
\newline
(iii) 	After writing our proofs, we realized that in \cite{CT}, a generalization of  Corollary~\ref{Corollary} for $k$-dependent theories has been proven. That proof is based on an extensive discussion of combinatorial topics. 
It is interesting to explore whether our method could 
 lead to a simpler proof for $k$-dependent theories (at least for classical logic). We will study this issue elsewhere.
\newline
(iv) It is worth mentioning another proofs (for the preservations of NIP and stability) by Ibarlucia \cite{I}, which applies only to the special case (i.e., $\aleph_0$-categorical theories) in terms of tame topological dynamics. }

\end{Remark}

\section*{Acknowledgments}
The authors would like to express their sincere gratitude to Kyle Gannon for pointing out a gap in the first version of this paper, which concerned the translation of certain results requiring the independence of random variables.

The first author would like to thank School of Mathematics, Institute for Research in Fundamental Sciences  (IPM), P. O. Box 19395-5746, Tehran, Iran, for their support. 
This research was in part supported by a grant from IPM (No. 1404030118).

\subsection*{Use of AI Tools}
The authors used the free version of ChatGPT 5.6 to assist with the cross-block Ramsey lemma and the proof of the stability section, and for editing and error-checking earlier drafts. All ideas and proofs in the $NIP$ and $NSOP$ sections, as well as the direction of the stability proofs, are the authors’ own.

\let\oldthebibliography\thebibliography
\renewcommand{\thebibliography}[1]{\oldthebibliography{#1}%
  \setlength{\itemsep}{0pt}\setlength{\parsep}{0pt}\setlength{\parskip}{0pt}}

\end{document}